\documentclass[11pt,reqno]{amsart}
\usepackage{amsmath, amsfonts, amsthm, amssymb,color}

\usepackage{setspace}
\usepackage{amsthm,amsmath,amssymb}
\usepackage{mathrsfs}
\usepackage{geometry}
\usepackage{amsfonts}
\usepackage[utf8]{inputenc}
\usepackage{amssymb}
\usepackage{amsmath}
\usepackage{subfigure}
\usepackage{graphicx}
\usepackage{amsmath,amscd}
\everymath{\displaystyle}
\usepackage{indentfirst} 
\usepackage{enumerate}
\usepackage[colorlinks]{hyperref}
\usepackage{bm}
\usepackage{algorithm}
\usepackage{algorithmic}
\usepackage{mathrsfs}
\usepackage[toc,page,title,titletoc,header]{appendix}
\usepackage{fancyhdr}

\numberwithin{equation}{section}

\newtheorem{thm}{Theorem}[section]
\newtheorem{lemma}[thm]{Lemma}
\newtheorem{cor}[thm]{Corollary}
\newtheorem{prop}[thm]{Proposition}
\newtheorem{rmk}[thm]{Remark}

\newtheorem{defi}[thm]{Definition}
\newtheorem{claim}[thm]{Claim}

\usepackage{graphicx} 
\numberwithin{equation}{section}

\newcommand{\ud}{\,\mathrm{d}}

\newcommand{\eps}{\varepsilon}
\newcommand{\vphi}{\varphi}
\newcommand{\R}{{\mathbb{R}}}
\newcommand{\M}{{\mathcal{M}}}
\newcommand{\C}{{\mathcal{C}}}

\newcommand{\ml}{{\mathcal{L}}}
\newcommand{\mm}{{\mathcal{M}}}
\newcommand{\me}{{\mathcal{E}}}

\def\Id{\operatorname{Id}}
\def\p{\partial}

\definecolor{darkgreen}{rgb}{0.0, 0.42, 0.0}

\author[JA Carrillo]{José Antonio Carrillo} \address{Mathematical Institute, University of Oxford, Oxford, OX2 6GG, UK.}
\email{carrillo@maths.ox.ac.uk.}

\author[X Feng]{Xuanrui Feng}
\address{School of Mathematical Sciences, Peking University, Beijing 100871, China.}
\email{pkufengxuanrui@stu.pku.edu.cn}

\author[S Guo]{Shuchen Guo}
\address{Mathematical Institute, University of Oxford, Oxford, OX2 6GG, UK.}
\email{guo@maths.ox.ac.uk}

\author[PE Jabin]{Pierre-Emmanuel Jabin} \address{Department of Mathematics and Huck Institutes, Pennsylvania State University, State College, PA 16801, USA.}
\email{pejabin@psu.edu.}

\author[Z Wang]{Zhenfu Wang}
\address{Beijing International Center for Mathematical Research, Peking University, Beijing 100871, China}
\email{zwang@bicmr.pku.edu.cn}

\title[Particle Approximation of the Landau Equation]
{Relative Entropy Method for Particle Approximation of the Landau Equation for Maxwellian Molecules}
\date{\today}

\begin{document}

\begin{abstract}
We derive the spatially homogeneous Landau equation for Maxwellian molecules from a natural stochastic interacting particle system. More precisely, we control the relative entropy between the joint law of the particle system and the tensorized law of the Landau equation. To obtain this, we establish as key tools the pointwise logarithmic gradient and Hessian estimates of the density function and also a new Law of Large Numbers result for the particle system. The logarithmic estimates are derived via the Bernstein method and the parabolic maximum principle, while the Law of Large Numbers result comes from crucial observations on the control of moments at the particle level.

\end{abstract}

\maketitle


\section{Introduction}

We consider the spatially homogeneous Landau equation for Maxwellian molecules in dimension $d$ in the following form
\begin{equation}\label{equation}
   \begin{cases}
\frac{\partial f}{\partial t}=Q(f,f)=\frac{\partial}{\partial v_\alpha}\int_{\mathbb{R}^d} a_{\alpha \beta}(v-v_\ast)\Big(f(v_\ast)\frac{\partial f(v)}{\partial v_\beta}-f(v)\frac{\partial f(v_\ast)}{\partial v_{\ast \beta}}\Big) \ud v_\ast,\\
f(0)=f_{0}, \end{cases}
\end{equation}
where $t \geq 0$ and $v \in \R^d$, and we shall always adopt the Einsten's summation convention with $\alpha$ and $\beta$ running from 1 to  $d$.  Hereinafter we shall call it the Landau-Maxwellian equation. The nonlinear operator $Q$ is referred as the Landau collision operator. The coefficient matrix appearing in the integrand of the operator $Q$ given by
\begin{equation*}\label{matrix}
a(z)=|z|^2\Pi(z),\quad \mbox{where the projection matrix  }\quad\Pi_{\alpha \beta}(z)=\delta_{\alpha \beta}-\frac{z_\alpha z_\beta}{|z|^2}.
\end{equation*}
This matrix is symmetric and non-negative (here $\delta_{\alpha \beta}$ is the  Kronecker delta).  The unknown function $f$ represents the density in velocity of gas molecules at time $t\geq 0$ in a plasma undergoing short-range interactions according to the Maxwellian potential. We shall always assume that $f$ is a non-negative density function with finite moments up to second-order in order to have finite kinetic energy.

The celebrated Landau equation \cite{landau1936kinetische} has been widely used in plasma physics as it can be derived from the Boltzmann equation \cite{boltzmann1970weitere,boltzmann2022lectures} through the grazing collision limit. The most physically meaningful Landau equation should consider Coulomb interactions. In this work, we instead focus on the Maxwellian molecules case due to its simpler structure as pointed out for example in \cite{villani1998spatially}. 

\subsection{Notations}

We gather here some global notations throughout this article. The Landau-Maxwellian equation conserves the total mass, momentum and kinetic energy. Hence without loss of generality,  we normalize the quantities to
\begin{equation}\label{conservation}
    \int_{\R^d} f(v) \ud v=1, \quad \int_{\R^d} vf(v) \ud v=0, \quad \int_{\R^d} |v|^2 f(v) \ud v=d,
\end{equation}
where we always use $f(t,v)$ to represent the solution to the  Landau-Maxwellian equation \eqref{equation}. We can also rewrite \eqref{equation} into a compact formulation as
\begin{equation}\label{Landau}
   \p_t f=\nabla \cdot [(a\ast f)\nabla f-(b\ast f)f], \tag{1.1a}
\end{equation}
or in a non-divergence form as
\begin{equation}\label{Landau1}
   \p_t f=(a\ast f):\nabla^2 f - (c\ast f)f, \tag{1.1b} 
\end{equation}
with the vector field $b=\nabla\cdot a$ and the scalar $c=\nabla \cdot b$, namely
\begin{equation*}
    b_\alpha(v)=\partial_\beta a_{\alpha\beta}(v)=-(d-1)v_\alpha, \quad c(v)=\partial_\alpha b_\alpha(v)=-d(d-1),
\end{equation*}
which will also appear in the simplified formulation of the Landau collision operator in Section \ref{formulation}. 

\subsection{Particle System}

We are interested in deriving the Landau-Maxwellian equation as the mean-field limit of several many-particle systems, via the classical strategy, namely the {\em propagation of chaos} argument. Our starting point is the following interacting particle system of $N$ indistinguishable particles, proposed for instance by Fournier \cite{fournier2009particle}: for $i=1, 2, \cdots, N$, 
\begin{equation}\label{SDE}
 \begin{cases} & \ud V_t^i=\frac{2}{N}\sum_{j=1}^{N}b(V_t^i-V_t^j)\ud t+\sqrt{2}\Big(\frac{1}{N}\sum_{j=1}^{N}a(V_t^i-V_t^j)\Big)^{\frac{1}{2}}\ud B_t^i, \\ & V_0^i=\zeta^i,   
 \end{cases}
\end{equation}
where the $(B_{\cdot}^i)_{i\geq 1}$ denote $N$ independent copies of standard $d$-dimensional Brownian motions which are independent with the initial data $\zeta^i$. The diffusion coefficient matrix is given by the square root of the non-negative symmetric matrix.

In our setting for the Landau-Maxwellian equation, we have that $a(0)=0$ and $b(0)=0$ hence the self-interaction is automatically 0. We thus use the convention $\textstyle\sum_{j=1}^N$  in \eqref{SDE} instead of the one  $\textstyle\sum_{j \ne i}$. The same formulation of particle systems has been considered in some previous works for instance  \cite{fournier2009particle,carrillo2024mean}.  Under the assumption of Maxwellian molecules, the existence and uniqueness of the strong solution of the system \eqref{SDE} have been proved in \cite{fournier2009particle}.
Applying the It\^{o}'s formula and using the relation $\nabla\cdot a=b$, we write down the Liouville equation (Landau master  equation) of the $N$-particle joint distribution $F_N(t, V)$, where $V=(v^1,\ldots,v^N) \in \R^{dN}$,  as 
\begin{equation}\label{master}
\begin{aligned}
\p_t F_N= &\, \sum_{i=1}^{N}\nabla_{v^{i}}\cdot\Big[\frac{1}{N}\sum_{j=1}^{N}a(v^{i}-v^{j})\nabla_{v^{i}}F_N-\frac{1}{N}\sum_{j=1}^{N}b(v^{i}-v^{j})F_N\Big],
\end{aligned}
\end{equation}
where the initial data is denoted by $F_N(0)$. And we also define the $k$-marginal of the joint distribution $F_{N}$, which is simply the joint distribution for the first $k$ particles, thanks to the exchangeability, as
\begin{equation*}
    F_{N,k}(t,v^1,...,v^k)=\int_{\R^{d(N-k)}} F_N(t,V) \ud v^{k+1}...\ud v^N.
\end{equation*}
\begin{rmk}
To prove the Law of Large Numbers Theorem in Section \ref{poc}, we need to assume that $F_N(0)$ is symmetric and close to the tensorized law $f_0^{\otimes N}$. To be more precise, if we let the test functions $\vphi_k$ ($k=1,2,3$) be any of $1,v_\alpha, v_\alpha^2, v_\alpha v_\beta, |v|^2$, where $\alpha\neq\beta$, then it requires that the third-marginal of the joint distribution  $F_{N,3}$ satisfies that
$$
\bigg|\int_{\R^{3d}}F_{N,3}(0)\vphi_1(v^1)\vphi_2(v^2)\vphi_3(v^3)\ud v^1\ud v^2\ud v^3-\prod_{k=1}^3\int_{\R^d}f_0(v^k)\vphi_k(v^k)\ud v^k\bigg|\leq \frac{C}{N}
$$ 
for some constant $C$. For simplicity, we can brutally assume that the initial data of \eqref{SDE} are i.i.d., namely $F_N(0)=f^{\otimes N}_0$.
\end{rmk}

We can formally obtain a first {\em a priori} estimate by multiplying the Liouville equation \eqref{master}
with $F_N$ and integrating over the domain, say
\begin{equation*}\label{L2 apriori}
\begin{aligned}
\frac{1}{2}\frac{\ud}{\ud t}\int_{\R^{dN}}F_N^2\ud V+ &\, \sum_{i=1}^{N}\int_{\R^{dN}}\frac{1}{N}\sum_{j=1}^{N}a(v^{i}-v^{j}):(\nabla_{v^{i}}F_N)^{\otimes 2}\ud V\\&-\sum_{i=1}^{N}\int_{\R^{dN}}\nabla_{v^{i}}F_N\cdot \frac{1}{N}\sum_{j=1}^{N}b(v^{i}-v^{j})F_N\ud V=0.
\end{aligned}
\end{equation*}
Together with the identity given by integration by parts
$$
\begin{aligned}
-\int_{\R^{dN}}\nabla_{v^{i}}F_N\cdot \frac{1}{N}\sum_{j=1}^{N}b(v^{i}-v^{j})F_N\ud V&=\frac{1}{2}\int_{\R^{dN}}F_N^2\frac{1}{N}\sum_{j=1}^{N}\nabla\cdot b(v^{i}-v^{j})\\&=-\frac{d(d-1)}{2}\int_{\R^{dN}}F_N^2\ud V,
\end{aligned}
$$
we obtain that
\begin{equation}\label{L2apriori}
\begin{aligned}
&\frac{1}{2}\frac{\ud}{\ud t}\int_{\R^{dN}}F_N^2\ud V+  \sum_{i=1}^{N}\int_{\R^{dN}}\frac{1}{N}\sum_{j=1}^{N}a(v^{i}-v^{j}):(\nabla_{v^{i}}F_N)^{\otimes 2}\ud V\\=& \frac{Nd(d-1)}{2}\int_{\R^{dN}}F_N^2\ud V.
\end{aligned}
\end{equation}
Similarly, multiplying the Liouville equation \eqref{master} with $\log F_N$ and integrating over the domain, we obtain that
\begin{equation}\label{entropyapriori}
\begin{aligned}
&\frac{\ud}{\ud t}\int_{\R^{dN}} F_N\log F_N \ud V + \int_{\R^{dN}}F_N\frac{1}{N}\sum_{i,j=1}^N a(v^i-v^j):(\nabla_{v^i} \log F_N)^{\otimes 2} \ud V = Nd(d-1). 
\end{aligned}        
\end{equation}
The well-posedness of the particle system \eqref{SDE} implies the existence of some measure-valued solution of the Liouville equation \eqref{master}. Combining with the {\em a priori} estimates above, we state the following proposition of existence of the weak solutions. 
\begin{prop}[Existence of Weak Solutions of the Liouville Equation]\label{weaksolution}
    Assume that the initial data $F_N(0)$ for the Liouville equation \eqref{master}  is tensorized, i.e. $F_N(0) = f_0^{\otimes N}$, with  $f_0\in L^1\cap L^2(\R^d)$ satisfying  that 
    $$
    f_0\geq 0,\quad \int_{\R^d}f_0 \ud v=1,\text{ and }\int_{\R^d}f_0\log f_0 \ud v <\infty.
    $$
Then for any $T>0$ and fixed $N$, there exists a weak solution $F_N$ of \eqref{master} satisfying $$F_N(t,V) \geq 0, \quad \int_{\R^{dN}}F_N(t,V) \ud V=1,$$
 $$
\|F_N(t)\|_{L^2(\R^{dN})}\leq e^{\frac{Nd(d-1)t}{2}}\|f_0\|^{N}_{L^2(\R^d)},
 $$ 
 and
 $$
 \int_{\R^{dN}}F_N(t)\log F_N(t) \ud V
\leq N\int_{\R^d}f_0\log f_0 \ud v +Nd(d-1)t
 $$
 for any $t \in [0,T]$.
\end{prop}
The $L^2$ bound and the entropy bound for $F_N$ are deduced from our {\em a priori} estimates \eqref{L2apriori} and \eqref{entropyapriori}. For more detailed discussion of the advection-diffusion theory, we refer to the book by Le Bris-Lions \cite{le2019parabolic}. In the appendix, we will also see that if $f_0$ has finite $p$-th order moment, then so does the first marginal of joint distribution $F_{N,1}$. Notice also that if $f_0$ has finite entropy, then $F_{N,1}$ has finite entropy by Jensen's inequality.
In our main result Theorem \ref{main},  we will further assume $f_0$ to be $\C^2$ with Gaussian type decay, which naturally satisfies the assumptions for the initial data in Proposition \ref{weaksolution} and Proposition \ref{moment estimate}. 

\begin{rmk}
     There is also another kind of particle system which has been introduced by Fontbona-Gu\'erin-M\'el\'eard \cite{fontbona2009measurability}, also in the sense of probabilistic interpretation, to approximate the Landau-Maxwellian equation, 
     \begin{equation}\label{SDE1}
 \begin{cases} & \ud V_t^i=\frac{2}{N} \sum_{j=1}^N b(V_t^i-V_t^j) \ud t +\frac{\sqrt{2}}{\sqrt{N}}\sum_{j=1}^N a(V_t^i-V_t^j)^{\frac{1}{2}} \ud B_t^{i,j} \\ & V_0^i=\zeta^i
 \end{cases}.
\end{equation}
Here $(B_t^{i,j})_{i,j \geq 1}$ are $N^2$ independent copies of standard $d$-dimensional Brownian motions and are also independent with the initial data $\zeta^i$. We notice that the particle system \eqref{SDE1} shares the same Liouville equation with the previous one \eqref{SDE}. Since we will always work on the Liouville equation level in this article, our results can also be applied to the system \eqref{SDE1}, once we make the same assumptions on the initial data.
\end{rmk}
\begin{rmk}
We also propose another particle system which has the same Liouville equation \eqref{master}, thus our results can also be applied to this case. The SDE system is given by
\begin{equation}\label{enviromental}
 \begin{cases}\ud V_t^i=\frac{2}{N}\sum_{j=1}^{N}b(V_t^i-V_t^j)\ud t+\frac{\sqrt{2}}{\sqrt{N}}\sum_{j=1}^{N}\sum_{\alpha<\beta}\xi_{\alpha,\beta}(V_t^i-V_t^j)\ud B_t^{j,\alpha,\beta} \\ V_0^i=\zeta^i   \end{cases} ,
\end{equation}
where the vector field $\xi_{\alpha,\beta}(v)=\big(0,\ldots,-v_\beta,\ldots, v_\alpha,\ldots 0\big)$ is given with two non-zero components at the $\alpha$-th and $\beta$-th  positions $(\alpha<\beta)$; $(B_t^{j,\alpha,\beta})_{j\geq 1, \alpha<\beta}$ are i.i.d. $1$-dimensional Brownian motions and independent with the initial data $\zeta^i$. Notice that the identity of the quadratic variation term holds: 
$$
\begin{aligned}
&\frac{1}{2}\ud\Big[\frac{\sqrt{2}}{\sqrt{N}}\sum_{j=1}^{N}\sum_{\alpha<\beta}\xi_{\alpha,\beta}(V_t^i-V_t^j)\ud B_t^{j,\alpha,\beta}\Big]_t\\=&\frac{1}{N}\sum_{j=1}^{N}\sum_{\alpha<\beta}\xi_{\alpha,\beta}\otimes\xi_{\alpha,\beta} \ud t=\frac{1}{N}\sum_{j=1}^{N}a(V_t^i-V_t^j)\ud t.
\end{aligned}
$$

The advantage of the formulation is that we write explicitly the martingale term without taking the square root of the corresponding matrices. For non-Maxwellian case, we replace the current vector field by $|v|^{\gamma/2}\xi_{\alpha,\beta}(v)$. This formulation of randomness is inspired by the environmental noise for instance in Guo-Luo \cite{guo2023scaling}.
\end{rmk}

\begin{rmk}
     Apart from the particle systems \eqref{SDE}, \eqref{SDE1} and \eqref{enviromental}, Carrapatoso \cite{carrapatoso2012propagation} has constructed another Kac-type particle system to approximate the Landau-Maxwellian equation, by passing the Kac's walk model to the grazing limit. We present here the weak form of the Liouville equation of this system.
     \begin{equation*}
    \partial_t F_N=\frac{2}{N}\sum_{i,j=1}^N b(v^i-v^j) \cdot \nabla_{v^i}F_N+\frac{1}{N}\sum_{i,j=1}^N a(v^i-v^j):(\nabla_{v^iv^i}^2F_N-\nabla_{v^iv^j}^2F_N).
\end{equation*}
Due to the conservation law of momentum and energy in the collision process of the Kac's walk, the new system also conserves the momentum and energy, which is different from the first two systems, formally leading to the appearance of the cross term $\nabla_{v^iv^j}^2 F_N$ in the Liouville equation. We expect a similar relative entropy estimate for this new system approximating the Landau-Maxwellian equation, but since the Liouville equation has a different form, we shall not deal with it in this work.
\end{rmk}

\subsection{Main Results}

The key idea in this article to establish propagation of chaos is to evaluate the normalized relative entropy defined below between the joint law of the particle system and the $N-$tensor product of the  solution of the limit equation,  bounding it with some quantities vanishing when  $N$ goes to infinity. Then by the classical  Csisz\'ar-Kullback-Pinsker inequality, the relative entropy estimate implies the quantitative  propagation of chaos in $L^1$ sense. 

\begin{defi}[Normalized Relative Entropy]\label{def RE}
    The normalized relative entropy between two probability measures on $\R^{dN}$, say  $F_N \text{ and } f_N \in \mathcal{P}(\R^{dN})$,  is defined by
    \begin{equation*}
        H_N(F_N|f_N)= \frac{1}{N} \int_{\R^{dN}} F_N \log \frac{F_N}{f_N} \ud V. 
    \end{equation*}
    We set the relative entropy quantity to be $+ \infty$ if $F_N$ is not absolutely continuous with respect to $f_N$. In this article, we take $F_N$ as solutions of the Liouville equation as defined before, and we take $f_N$ as 
    \begin{equation*}
        f_N(V):=f^{\otimes N} (v^1, \cdots, v^N) = \prod_{i=1}^N f(v^i), 
    \end{equation*}
    that is  the tensorized distribution of the solution to \eqref{equation}.
\end{defi}
Now we state our main result, where $|A|$ denotes the Frobenius norm of a matrix $A$.
\begin{thm}[Entropic Propagation of Chaos]\label{main}
For any $T>0$, given the classical solution $f \in \C^1([0,T],\C^2(\R^d))$ solving the Landau equation for Maxwellian molecules \eqref{equation} with non-negative initial data $f_0\in \C^2(\R^d)$ that satisfies the conservation laws  \eqref{conservation}. Assume further that there exists some positive constant $C_0$ such that $f_0$ satisfies the Gaussian-type bound from above that  \begin{equation}\label{inigauss}
        f_0(v) \leq C_0\exp(-C_0^{-1}|v|^2),
    \end{equation}
    and the logarithmic growth conditions
    \begin{equation}\label{inigradient}
        |\nabla \log f_0(v)|\lesssim1+|v|,
    \end{equation}
    \begin{equation}\label{inihessian}
        |\nabla^2 \log f_0(v)| \lesssim1+|v|^2.
    \end{equation}
Then for any weak solution of the Liouville  equation \eqref{master} with $F_N(0)=f^{\otimes N}_0$, we have the following estimate for $N$ large enough
$$
        H_N(F_N|f_N)(t) \leq \frac{C(1+T^{5/2})}{\sqrt{N}},
$$
where $C$ is a constant depending only on the initial bounds above and the initial 4-th order moment of the first marginal $F_{N,1}$.
\end{thm}
\begin{rmk}\label{classical solution} The well-posedness and regularity of the Landau-Maxwellian equation have been established in Villani \cite{villani1998spatially}. If $f_0$ is non-negative with finite mass and energy,  then  \eqref{equation} admits a unique classic solution which is bounded and in $C^\infty(\R^d)$ for all time $t>0$. 
\end{rmk}

Due to the sub-additivity of the relative entropy, we obtain that
\begin{equation*}
    H_k(F_{N,k}|f^{\otimes k}) := \frac{1}{k} \int_{\R^{dk}} F_{N,k}\log \frac{F_{N,k}}{f^{\otimes k}} \ud v^1 ... \ud v^k \leq H_N(F_N|f_N),
\end{equation*}
which in turn implies the strong $L^1$ propagation of chaos as follows from the classical Csisz\'ar-Kullback-Pinsker inequality, namely
$$
\|F_{N,k}-f^{\otimes k}\|_{L^1}\leq\sqrt{2kH_k(F_{N,k}|f^{\otimes k})}.
$$

\begin{cor}[Propagation of Chaos in $L^1$]\label{maincor}
    Under the same assumptions of Theorem \ref{main},  the propagation of chaos in $L^1$  holds with the convergence rate as
\begin{equation*}
        \Vert F_{N,k} -f^{\otimes k} \Vert_{L^\infty([0,T],L^1(\mathbb{R}^{dk}))} \leq \frac{C(1+T^{5/4})}{N^{1/4}}.
    \end{equation*}
\end{cor}

\begin{rmk}
   Our main results, Theorem \ref{main} and Corollary \ref{maincor}, can be compared to the result by Fournier in \cite{fournier2009particle}, which established the propagation of chaos in the sense of the Wasserstein-$2$ distance between the $k$-marginals $F_{N, k}(t)$ and the tensorized law $f_t^{\otimes k}$, that is,
\begin{equation}\label{FWEs}
\sup_{t\in[0,T]} W_2(F_{N,k},f_t^{\otimes k})\leq \frac{C_T}{N^{1/2}}, 
\end{equation}
where the constant $C_T$ depends on $T$ exponentially. Combining our relative entropy estimate and the Talagrand's inequality, see for instance in the book by Villani \cite{villani2021topics}, we can also obtain a convergence rate estimate in the Wasserstein-$2$ distance, but suboptimal compared to Fournier's result. On the other hand, if we replace our new argument in the flavor of the Law of Large Numbers in Section \ref{poc} simply by \eqref{FWEs}, we can also obtain our relative entropy estimate in the same order in $N$, but then the dependence on $T$ will be less optimal compared to our result.  
\end{rmk}

\subsection{Related Works}

The well-posedness and regularity results of the Landau equation have been deeply investigated, where we refer to Villani \cite{villani1998new,villani1998spatially} for physical and mathematical details of the equation and for the Maxwellian molecules case. In terms of hard potentials, the well-posedness, regularity and
long-time behavior have been studied by Desvillettes-Villani \cite{desvillettes2000spatially1,desvillettes2000spatially2} and Fournier-Heydecker \cite{fournier2021stability}. 
For moderately soft potentials, the global well-posedness result is obtained in Fournier-Guérin \cite{fournier2009well}; for very soft potentials. Villani \cite{villani1998new} defines the notion of $H$-solution and proves its existence, but the regularity and uniqueness of $H$-solutions remain open. Desvillettes \cite{desvillettes2015entropy} has proved that the $H$-solutions are actually weak solutions through an entropy dissipation estimate. Recently Golding-Loher \cite{golding2024local} proves the local existence, uniqueness and stability results for smooth solutions with $L^p$ initial data, where $p$ is arbitrarily close to $3/2$.  Most recently Guillen-Silvestre \cite{guillen2023landau} has shown that the classical solution of the Landau equation for Coulombian potential does not blow up. Combining with this argument, Golding-Gualdani-Loher \cite{golding2024global} has proved the global existence of bounded smooth solutions.

The derivation of the Landau equation from particle systems is interesting and challenging.  The first quantitative propagation of chaos result for the Landau equation was given by Fontbona-Gu\'erin-M\'el\'eard \cite{fontbona2009measurability} using optimal transportation techniques. Later in Fournier \cite{fournier2009particle} a better result in the decay rate of Wasserstein-2 distance was obtained via the classical coupling method. Both results work on the Landau-Maxwellian equation. Fournier-Hauray \cite{fournier2016propagation} was also able to obtain the quantitative result of the same particle system for the Landau equation with moderately soft potential using a weak-strong argument and the classical weak martingale approach introduced firstly by McKean. Carrapatoso \cite{carrapatoso2012propagation} firstly constructed the Kac-type particle system via passing the Kac's walk process to the grazing limit to approximate the Landau-Maxwellian equation, with a quite slowly decaying rate, using the abstract analytic method developed in  \cite{mischler2013kac}. This result was improved and extended to the hard potential case by Fournier-Guillin \cite{fournier2017kac}.

The relative entropy method to prove quantitative propagation of chaos result for McKean-Vlasov systems was first introduced by Jabin-Wang in \cite{jabin2016mean} for second-order systems with bounded kernels and in \cite{jabin2018quantitative} for general first-order systems with $W^{-1,\infty}$ kernels, including the point vortex model approximating the 2D Navier-Stokes equation on the torus.
Recently much progress has been made in extending the relative entropy method to more general cases and models, especially with singular interacting kernels. Those results include Wynter \cite{wynter2021quantitative} for the mixed-sign point vortex model, Guillin-Le Bris-Monmarch\'e \cite{guillin2024uniform} for pushing forward the results in \cite{jabin2018quantitative} to the uniform-in-time propagation of chaos by using the logarithmic Sobolev inequality (LSI) for the limit density, Shao-Zhao \cite{shao2024quantitative} for the general circulation case with common environmental noise approximating the stochastic 2D Navier-Stokes equation, and Carrillo-Guo-Jabin \cite{carrillo2024mean} for the derivation of the mean-field approximation for Landau-like equations, all of which are working on the torus case. In the recent work Feng-Wang \cite{feng2023quantitative}, the propagation of chaos result for 2D point vortex models in \cite{jabin2018quantitative} has been extended to the whole space case, by carefully investigating regularity results of the limit density, say growth estimates of $\nabla \log \bar\rho$ and $\nabla^2 \log \bar\rho$, once given some growth estimates on the same quantities of the initial data $\bar \rho_0$. These estimates are expected to be universal in many models with some Gaussian type equilibrium. We also remark that those assumptions on the initial data are reasonable, since they will automatically be recovered at any positive time given initial data with Gaussian decay. We further mention the work Huang \cite{huang2023entropy} which shows that for weakly interacting diffusions of the McKean-Vlasov type with bounded interaction kernels,  given weak chaotic initial data  for instance in the Wasserstein metric sense, one can obtain entropic propagation of chaos for $t >0$.  

Recently the modulated energy method introduced by Serfaty \cite{serfaty2020mean} has also seen great effectiveness in proving mean-field limit and propagation of chaos of McKean-Vlasov equations with singular interacting kernels. Instead of focusing on the Liouville equation level as in the relative entropy method, this modulated energy method works on the empirical measure of the particle system. By introducing and estimating a Coulomb/Riesz-based metric between the empirical measure and the limit density, which can be viewed as a renormalized negative-order Sobolev norm and is called the modulated energy, \cite{serfaty2020mean} arrives at proving the quantitative convergence rate of the empirical measure, acting on some test function, for the Coulomb and super-Coulomb case without additive noise in the whole Euclidean space. Later the method was extended to Riesz-type singular flows by Nguyen-Rosenzweig-Serfaty \cite{nguyen2022mean} and global-in-time result by Rosenzweig-Serfaty \cite{rosenzweig2023global}. Combining the modulated energy with the relative entropy multiplying with the viscosity, Bresch-Jabin-Wang \cite{bresch2019mean,bresch2019modulated,bresch2023mean} developed the modulated free energy method to deal with the attractive singular kernel case with additive noise, in particular the 2D Patlak-Keller-Segel model. De Courcel-Rosenzweig-Serfaty \cite{de2023sharp,de2023attractive} also extended the modulated free energy method to the periodic Riesz-type flow and the attractive log gas, proving uniform-in-time propagation of chaos results.

We also mention the recent work of Lacker \cite{lacker2021hierarchies}, where the local relative entropy of any order, say $H_k (F_{N, k}(t) \vert f_t^{\otimes k})$, is considered instead of the global quantity $H_N(F_N (t)\vert f_t^{\otimes N})$. Using the BBGKY hierarchy and some careful iteration, Lacker established the optimal convergence rate of the relative entropy between the $k-$marginals and the tensorized law for weakly interacting diffusions with bounded interaction kernels. Combining with the analysis in \cite{jabin2018quantitative}, Wang \cite{wang2024sharp} has extended the sharp convergence rate in $N$ to the 2D Navier-Stokes case on torus, but restricted to the high viscosity case. Bresch-Jabin-Soler \cite{bresch2022new} also works on the BBGKY hierarchy, combining with the compactness argument, to derive the mean-field limit of the second-order singular interacting Vlasov-Fokker-Planck system, in particular the 2D Coulomb case. Finally, we also mention the new approach in Bresch-Duerinckx-Jabin \cite{breschduerjab} that also relies on estimates on a hierarchy, based this time on a new notion of dual cumulants.

\subsection{Outline of the Article}

The rest of this article is organized as follows. In Section \ref{evolution} we derive the evolution of the relative entropy, and prove the main theorem by taking advantage of the logarithmic gradient and Hessian estimate and a new Law of Large Numbers theorem, which will be proved in later sections. We give the analysis of the Landau-Maxwellian collision operator and provide some discussions on the parabolic maximum principle in Section \ref{formulation}. Then we present our main regularity estimates, namely the logarithmic gradient and Hessian estimates in Section \ref{estimate}.  Section \ref{poc} is devoted to the proof of the Law of Large Numbers result. We leave the moment estimate of the Liouville equation in Appendix \ref{appendix}.  

\section{Proof of the Main Result}\label{evolution}
In this section, we will first prove our main result, i.e. the entropic propagation of chaos result Theorem \ref{main}, while admitting some intermediate results. It is given firstly by deriving the evolution of the normalized relative entropy, and then exploiting the Law of Large Numbers Theorem \ref{LLN} and logarithmic gradient and Hessian estimate Theorem \ref{loggradient} and Theorem \ref{loghessian}, to control the time derivative of the relative entropy. Integrating in time will conclude our entropy estimate result.
\begin{prop}[Evolution of the Relative Entropy]\label{evolution prop}
For any $t\in[0,T]$, it holds that
\begin{equation}\label{evolution RE}
\begin{aligned}
  H_N(F_N|f_N)(t)
  &\leq  H_N(F_N|f_N)(0)\\&-\frac{1}{N}\sum_{i=1}^{N}\int_{0}^{t}\int_{\R^{dN}} F_N\Big[a\ast f(v^{i})-\frac{1}{N}\sum_{j=1}^{N}a(v^{i}-v^{j})
\Big]: \frac{\nabla^2_{v^i}f_N}{f_N}\ud V\ud s.
\end{aligned}
\end{equation}  
\end{prop}
\begin{proof}
  Recall from Definition \ref{def RE} that
$$
H_N(F_N|f_N)=\frac{1}{N} \int_{\R^{dN}} F_N \log F_N \ud V-\frac{1}{N} \int_{\R^{dN}} F_N \log f_N \ud V,
$$
and the tensorized limit equation \eqref{Landau} satisfies
\begin{equation}
   \p_t f_N=\sum_{i=1}^N\nabla_{v^i} \cdot \big[a\ast f(v^i) \cdot \nabla_{v^i} f_N-b\ast f(v^i)f_N\big]. 
\end{equation}
We calculate the time derivative of the relative entropy as
$$
\begin{aligned}
\frac{\ud}{\ud t}\left(\frac{1}{N}\int_{\R^{dN}}F_N\log F_N\ud V\right)=&\,\frac{1}{N}\int_{\R^{dN}}(1+\log F_N)\p_t F_N\ud V
\\=&-\frac{1}{N}\sum_{i=1}^{N}\int_{\R^{dN}}\frac{1}{N}\sum_{j=1}^{N}a(v^{i}-v^{j}):\frac{\nabla_{v^{i}}F_N\otimes\nabla_{v^{i}}F_N}{F_N}\ud V
\\&+\frac{1}{N}\sum_{i=1}^{N}\int_{\R^{dN}}\frac{1}{N}\sum_{j=1}^{N}b(v^{i}-v^{j})\cdot\nabla_{v^{i}}F_N\ud V,
\end{aligned}
$$
where we plug in \eqref{master} and integrate by parts in the second step; similarly, we have
$$
\begin{aligned}
&\frac{\ud}{\ud t}\left(\frac{1}{N}\right.\left.\!\!\int_{\R^{dN}}F_N\log f_N\ud V\right)
=\frac{1}{N}\int_{\R^{dN}}\Big(\log f_N \p_t F_N + F_N \frac{ \p_t f_N}{f_N}\Big)\ud V
\\=&\,-\frac{1}{N}\sum_{i=1}^{N}\int_{\R^{dN}}\!\!\frac{1}{N}\sum_{j=1}^{N}a(v^{i}-v^{j}):
\frac{\nabla_{v^{i}}f_N\otimes\nabla_{v^{i}}F_N}{f_N} \ud V\\&\,+\frac{1}{N}\sum_{i=1}^{N}\int_{\R^{dN}}\!\!\frac{1}{N}\sum_{j=1}^{N}b(v^{i}-v^{j})\cdot\frac{\nabla_{v^{i}}f_N}{f_N}F_N\ud V
\\&\,-\frac{1}{N}\sum_{i=1}^{N}\int_{\R^{dN}} a\ast f(v^{i}) :\nabla_{v^{i}}\frac{F_N}{f_N}\otimes\nabla_{v^{i}}f_N\ud V+\frac{1}{N}\sum_{i=1}^{N}\int_{\R^{dN}}b\ast f(v^{i})\cdot \nabla_{v^{i}}
\frac{ F_N }{f_N} f_N\ud V.
\end{aligned}
$$
Using the identity 
$$
\nabla_{v^i} \frac{F_N}{f_N}=\frac{f_N \nabla_{v^i} F_N-F_N \nabla_{v^i} f_N}{f_N^2},
$$
we are able to rewrite 
\begin{equation}\label{RE}
\begin{aligned}
&\,\frac{\ud}{\ud t}H_N(F_N|f_N)(t) 
\\=&\, -\frac{1}{N}\sum_{i=1}^{N}\int_{\R^{dN}}F_N\frac{1}{N}\sum_{j=1}^{N}a(v^{i}-v^{j}):\nabla_{v^{i}}\log\frac{F_N}{f_N}\otimes\nabla_{v^{i}}\log\frac{F_N}{f_N}\ud V\\
 &\, +\frac{1}{N}\sum_{i=1}^{N}\int_{\R^{dN}} F_N\Big[a\ast f(v^{i})-\frac{1}{N}\sum_{j=1}^{N}a(v^{i}-v^{j})
\Big]:\nabla_{v^{i}}\log\frac{F_N}{f_N}\otimes\nabla_{v^{i}}\log f_N\ud V
\\ &\, -\frac{1}{N}\sum_{i=1}^{N}\int_{\R^{dN}}F_N\Big[b\ast f(v^{i})-\frac{1}{N}\sum_{j=1}^{N}b(v^{i}-v^{j})\Big]\cdot \nabla_{v^{i}}\log
\frac{ F_N }{f_N} \ud V\\
=&:\, I_1+I_2+I_3.
\end{aligned}
\end{equation}

For the first term $I_1$, due to the idempotency of the projection matrix, i.e.  $\Pi^2=\Pi$, we have 
$$
\begin{aligned}
I_1=&-\frac{1}{N^2}\sum_{i,j=1}^{N}\int_{\R^{dN}}F_N|v^{i}-v^{j}|^{2}\big(\Id-\frac{(v^{i}-v^{j})\otimes(v^{i}-v^{j})}{|v^{i}-v^{j}|^2}\big):\Big(\nabla_{v^{i}}\log\frac{F_N}{f_N}\Big)^{\otimes2}\ud V\\=&-\frac{1}{N^2}\sum_{i=1}^{N}\sum_{j=1}^{N}\int_{\R^{dN}}F_N|v^{i}-v^{j}|^{2}\Big|\big(\Id-\frac{(v^{i}-v^{j})\otimes(v^{i}-v^{j})}{|v^{i}-v^{j}|^2}\big)\nabla_{v^{i}}\log\frac{F_N}{f_N}\Big|^2\ud V\leq 0,\\
\end{aligned}
$$
which is non-positive. For $I_3$, it holds from integrating by parts that
$$
\begin{aligned}
I_3=&-\frac{1}{N}\sum_{i=1}^{N}\int_{\R^{dN}}F_N\Big[b\ast f(v^{i})-\frac{1}{N}\sum_{j=1}^{N}b(v^{i}-v^{j})\Big]\cdot \nabla_{v^{i}}\log
\frac{ F_N }{f_N} \ud V \\ =&   \frac{1}{N}\sum_{i=1}^{N}\int_{\R^{dN}}\Big[a\ast f(v^{i})-\frac{1}{N}\sum_{j=1}^{N}a(v^{i}-v^{j})\Big]:\nabla_{v^i}\Big[F_N \nabla_{v^{i}}\log
\frac{ F_N }{f_N} \Big]\ud V\\ =&   \frac{1}{N}\sum_{i=1}^{N}\int_{\R^{dN}}F_N\Big[a\ast f(v^{i})-\frac{1}{N}\sum_{j=1}^{N}a(v^{i}-v^{j})\Big]:\nabla_{v^i}\log F_N \otimes\nabla_{v^{i}}\log
\frac{ F_N }{f_N} \ud V\\ &+   \frac{1}{N}\sum_{i=1}^{N}\int_{\R^{dN}}F_N\Big[a\ast f(v^{i})-\frac{1}{N}\sum_{j=1}^{N}a(v^{i}-v^{j})\Big]:\nabla_{v^{i}}^2\log
\frac{ F_N }{f_N} \ud V.
\end{aligned}
$$
The following trivial identity holds:
$$
\nabla_{v^{i}}^2\log
\frac{ F_N }{f_N}=\frac{\nabla_{v^{i}}^2F_N}{F_N}-\frac{\nabla_{v^{i}}^2f_N}{f_N}-\nabla_{v^{i}}\log\frac{F_N}{f_N}\otimes \nabla_{v^{i}}\log (F_Nf_N),
$$
which yields
$$
\begin{aligned}
I_2+I_3=&  \frac{1}{N}\sum_{i=1}^{N}\int_{\R^{dN}}F_N\Big[a\ast f(v^{i})-\frac{1}{N}\sum_{j=1}^{N}a(v^{i}-v^{j})\Big]:\Big(\frac{\nabla_{v^{i}}^2F_N}{F_N}-\frac{\nabla_{v^{i}}^2f_N}{f_N}\Big) \ud V  \\
=& \frac{1}{N}\sum_{i=1}^{N}\int_{\R^{dN}}\Big[a\ast f(v^{i})-\frac{1}{N}\sum_{j=1}^{N}a(v^{i}-v^{j})\Big]:\nabla_{v^{i}}^2F_N \ud V \\&-
\frac{1}{N}\sum_{i=1}^{N}\int_{\R^{dN}}F_N\Big[a\ast f(v^{i})-\frac{1}{N}\sum_{j=1}^{N}a(v^{i}-v^{j})\Big]:\frac{\nabla_{v^{i}}^2f_N}{f_N} \ud V
\\=&
-\frac{1}{N}\sum_{i=1}^{N}\int_{\R^{dN}}F_N\Big[a\ast f(v^{i})-\frac{1}{N}\sum_{j=1}^{N}a(v^{i}-v^{j})\Big]:\frac{\nabla_{v^{i}}^2f_N}{f_N} \ud V,
\end{aligned}
$$
where the last step is obtained by integrating by parts twice and using that $\nabla^2:a=-d(d-1)$ in the Landau-Maxwellian case. Proposition \ref{evolution prop} is obtained by integrating with respect to time.
\end{proof}

In order to control the last term in Proposition \ref{evolution prop}, one may recall the Law of Large Numbers and the Large Deviation type results obtained in \cite[Theorem 3, Theorem 4]{jabin2018quantitative}. As observed in \cite{feng2023quantitative}, it will work once the integrand satisfies a quadratic growth estimate. However, due to the unboundedness of the coefficient matrix $a(z)$, it is not the case this time. Hence we shall apply the Cauchy-Schwarz inequality to the three-index variables as $$
\sum_{i=1}^N\sum_{\alpha,\beta=1}^d x_{i\alpha\beta} y_{i\alpha\beta}\leq\Big(\sum_{i=1}^N\sum_{\alpha,\beta=1}^d x_{i\alpha\beta}^2\Big)^{\frac{1}{2}}\Big(\sum_{i=1}^N\sum_{\alpha,\beta=1}^d y_{i\alpha\beta}^2\Big)^{\frac{1}{2}}
$$
to the integrand  of \eqref{evolution RE}, which yields that 
\begin{equation}\label{Holder}
\begin{aligned}
&\left|\frac{1}{N}\sum_{i=1}^N\int_{\R^{dN}} F_N\Big[a\ast f(v^{i})-\frac{1}{N}\sum_{j=1}^{N}a(v^{i}-v^{j})
\Big]: \frac{\nabla^2_{v^i}f_N}{f_N}\ud V\right|\\
\leq &\int_{\R^{dN}} F_N\Big(\frac{1}{N}\sum_{i=1}^{N}\Big|a\ast f(v^{i})-\frac{1}{N}\sum_{j=1}^{N}a(v^{i}-v^{j})
\Big|^2\Big)^{\frac{1}{2}}\Big(\frac{1}{N}\sum_{i=1}^{N}\Big|\frac{\nabla^2 f}{f}(v^i)\Big|^2\Big)^{\frac{1}{2}}\ud V\\
\leq &\bigg(\int_{\R^{dN}} F_N\frac{1}{N}\sum_{i=1}^{N}\Big|a\ast f(v^{i})-\frac{1}{N}\sum_{j=1}^{N}a(v^{i}-v^{j})
\Big|^2\bigg)^{\frac{1}{2}}\bigg(\int_{\R^{dN}} F_N\frac{1}{N}\sum_{i=1}^{N}\Big|\frac{\nabla^2 f}{f}(v^i)\Big|^2\bigg)^{\frac{1}{2}},
\end{aligned}
\end{equation}
where the last inequality is due to the H\"{o}lder's inequality. 

Now we can treat the two parts in the last line separately. For the first part, we observe that the quantity inside the square term has mean zero if the joint law of $(v^1,...,v^N)$ equals to the tensorized law $f_N$. This will be shown by the following Law of Large Numbers result, which can be viewed as an extension to unbounded kernel case of the Law of Large Numbers type result \cite[Theorem 3]{jabin2018quantitative}.

\begin{thm}[Law of Large Numbers]\label{LLN}
Assume that $T\sim O(N)$, we have the following Law of Large Numbers estimate:
\begin{equation}\label{cancel estimate}
\int_{\R^{dN}} F_N\frac{1}{N}\sum_{i=1}^{N}\Big|a\ast f(v^{i})-\frac{1}{N}\sum_{j=1}^{N}a(v^{i}-v^{j})
\Big|^2\ud V \lesssim \frac{(1+\mm_4(0))(1+T)}{N}.
\end{equation}
\end{thm}   
The proof of Theorem \ref{LLN} will be postponed to Section \ref{poc}, but we emphasize here that, the reason why we keep the structure $\textstyle\frac{1}{N}\sum_i$ instead of using some fixed index $i$ is to see the symmetry property easily. To deal with the second part in the last line of \eqref{Holder}, we need to take advantage of the logarithmic gradient and Hessian estimates and bound it with the moment of $F_N$. For the completeness of the proof, we present here two key estimates.

\begin{thm}[Logarithmic Gradient Estimate]\label{loggradient}
Assume that the solution of the   \, \,
Landau-Maxwellian equation $f \in \C^1([0,T],\C^2(\R^d))$ with its initial data
    satisfying the logarithmic growth bound \eqref{inigradient}  
        $$
        |\nabla \log f_0| \lesssim1+|v|;
$$
we further assume that the initial data can be controlled above by some Gaussian-type function with some constant $C_0>0$ as \eqref{inigauss}
\begin{equation*}
        f_0 \leq C_0\exp(-C_0^{-1}|v|^2).
    \end{equation*}
Then for any $t\in[0,T]$ we have
    \begin{equation*}
        |\nabla \log f| \lesssim1+\sqrt{t}+|v|.
    \end{equation*}
\end{thm}

\begin{thm}[Logarithmic Hessian Estimate]\label{loghessian}
    Under the same assumptions as in Theorem \ref{loggradient} and further assume that $f_0$ satisfies the logarithmic Hessian estimate \eqref{inihessian}
    \begin{equation*}
        |\nabla^2\log f_0|
    \lesssim1+|v|^2,
    \end{equation*}
    then for any $t$ we have
    \begin{equation*}
        |\nabla^2 \log f| \lesssim1+t+|v|^2.
    \end{equation*}
\end{thm}
The proofs, which are similar to the logarithmic estimates in \cite{feng2023quantitative}, are presented in Section \ref{estimate}. These two estimates together tell us that the $\nabla^2f/f$ has the quadratic growth bound in velocity, say
$$
\Big|\frac{\nabla^2 f}{f}(v^i,t)\Big|\leq \big|\nabla^2 \log f(v^i,t)\big|+\Big|\nabla \log f(v^i,t)\Big|^2\lesssim1+t+|v^i|^2.
$$
By the moment estimate Proposition \ref{moment estimate}, for $t\in[0,T]$, the second part can finally be bounded by some constant only depends on $T$ and the initial 4-th moment $\M_4(0)$ as
\begin{equation}\label{fourth order}
\begin{aligned}
\sup_{t\in[0,T]}\int_{\R^{dN}} F_N\frac{1}{N}\sum_{i=1}^{N}\Big|\frac{\nabla^2 f}{f}(v^i,t)\Big|^2\ud V\lesssim&\int_{\R^{dN}}F_N\frac{1}{N}\sum_{i=1}^N(1+t+|v^i|^2)^2\ud V\\\lesssim& 
 1+T^2+\mm_4(0)e^{\frac{8T}{N}}.
\end{aligned}
\end{equation}
We plug estimates  \eqref{cancel estimate} and \eqref{fourth order} above into \eqref{Holder} to  obtain that for $T<N$,
$$
\sup_{t\in[0,T]}\bigg|\frac{1}{N}\sum_{i=1}^N\int_{\R^{dN}} F_N\Big[a\ast f(v^{i})-\frac{1}{N}\sum_{j=1}^{N}a(v^{i}-v^{j})
\Big]: \frac{\nabla^2_{v^i}f_N}{f_N}\ud V\bigg|\lesssim \frac{(1+\mm_4(0))(1+T^\frac{3}{2})}{\sqrt{N}}.
$$
By the evolution of the relative entropy Proposition \ref{evolution RE}, together with our initial assumption $H_N(F_N|f_N)(0)=0$, we can conclude the main result Theorem \ref{main} as
$$
\begin{aligned}
  &H_N(F_N|f_N)(t)\\
  \leq\,&  \int_{0}^{t}\sup_{s\in[0,T]}\bigg|\frac{1}{N}\sum_{i=1}^N\int_{\R^{dN}} F_N\Big[a\ast f(v^{i})-\frac{1}{N}\sum_{j=1}^{N}a(v^{i}-v^{j})
\Big]: \frac{\nabla^2_{v^i}f_N}{f_N}\ud V\bigg|\ud s\\
 \leq\,& \frac{C(1+\mm_4(0))(1+T^\frac{5}{2})}{\sqrt{N}}.
\end{aligned}
$$

\section{Maximum Principle}\label{formulation}

We shall present in this section the classical formulation of the Landau-Maxwellian equation. Briefly speaking, we rewrite the Landau collision operator in the form of a second-order elliptic operator, which is natural to satisfy some maximum principle. This will be the key method to derive our logarithmic estimates. 

\subsection{Formulation of the Landau-Maxwellian Equation}

The first step is to simplify the collision operator by making use of the conservation laws \eqref{conservation}, which is directly computed as in \cite[Section 2]{villani1998spatially}, say rewriting \eqref{Landau1} into
\begin{equation*}
    \partial_t f=\ml f=\Bar{a}_{\alpha\beta}\partial_{\alpha\beta}f+d(d-1)f,
\end{equation*}
where the matrix $(\bar a_{\alpha\beta})_{d\times d}$ is given by
\begin{equation}\label{bara}
\Bar{a}_{\alpha\beta}=a_{\alpha \beta} \ast f=d\delta_{\alpha \beta}+(|v|^2\delta_{\alpha \beta}-v_\alpha v_\beta)-\mathcal{E}_{\alpha \beta}(t),
\end{equation}
where $\mathcal{E}_{\alpha \beta}(t)$ is the directional temperature defined by
\begin{equation}\label{Eab}
\mathcal{E}_{\alpha \beta}(t):=\int_{\R^d}f(t,v)v_\alpha v_\beta \ud v.    
\end{equation}
Define the initial condition
\begin{equation*}
    D_{\alpha \beta}:=\int_{\R^d} f_0 v_\alpha v_\beta-\delta_{\alpha \beta}.
\end{equation*}
Without loss of generality, by choosing an appropriate orthonormal basis, we may diagonalize the matrix $(D_{\alpha \beta})_{d\times d}$ as 
\begin{equation*}
    (D_{\alpha \beta})= \text{ Diag } (D_{11},...,D_{dd})
\end{equation*}
with
\begin{equation*}
D_{\alpha \beta}=\int_{\R^d}f_0v_\alpha v_\beta =0,\mbox{ for } \alpha \neq \beta;\quad  D_{\alpha \alpha}=\int_{\R^d} f_0 v_\alpha^2-1,\text{ and } \sum_{\alpha=1}^d D_{\alpha \alpha}=0.
\end{equation*}
 For any smooth test function $\vphi(v)$,  the weak form of the collision operator $\ml$ is given by \begin{equation}\label{weak form}
\begin{aligned} & \int_{\R^d} \ml f(v)\vphi(v)  \ud v \\ &= \frac{1}{2} \sum_{\alpha=1}^d \sum_{\beta=1}^d \iint_{\R^d \times \R^d} f(v) f(w) a_{\alpha\beta}(v-w)\big(\partial_{\alpha\beta} \varphi(v)+\partial_{\alpha\beta} \varphi(w)\big) \ud v \ud w \\ & \quad+\sum_{\alpha=1}^d \iint_{\R^d \times \R^d}  f(v) f(w) b_\alpha(v-w)\big(\partial_\alpha \varphi(v)-\partial_\alpha \varphi(w)\big) \ud v \ud w .\end{aligned}
\end{equation}
By taking $\vphi$ as $1, v, |v|^2$ respectively, we can easily verify the conservation laws \eqref{conservation}; moreover, we can take $\vphi(v)=v_\alpha^2$ to obtain that
$$
\begin{aligned}
 \frac{\ud}{\ud t}\int_{\R^d} f(t,v)v_\alpha^2\ud v
 =&2\int_{\R^d \times \R^d} \big(|v-w|^2-(v-w)_\alpha^2\big)f(t,w)f(t,v)\ud w\ud v\\&-2(d-1)\int_{\R^d \times \R^d}(v-w)^2_\alpha f(t,w)f(t,v)\ud w\ud v\\
=&4\int_{\R^{d}}f(t,v)|v|^2\ud v-4d\int_{\R^{d}} f(t,v)v_\alpha^2\ud v ,
\end{aligned}
$$
which yields the explicit expression of the directional temperature \eqref{Eab} when $\alpha=\beta$
\begin{equation}\label{direction temp}
\mathcal{E}_{\alpha}(t) :=\mathcal{E}_{\alpha \alpha}(t)=1+(\mathcal{E}_{\alpha}(0)-1)e^{-4dt}=1+D_{\alpha\alpha}e^{-4dt}.    
\end{equation}
As mentioned in \cite{villani1998spatially,desvillettes2000spatially2}, under the assumption that $f_0 \in L^1$, it cannot hold that $D_{\alpha \alpha}=d-1$ or $D_{\alpha \alpha}=-1$, since the density function cannot be concentrated in some $(d-1)$-dimensional hyperplane. Hence there exists some $\eta>0$ depending on $f_0$, such that for any $\alpha$,
\begin{equation}\label{eta}
    -1+\eta \leq D_{\alpha \alpha} \leq d-1-\eta,
\end{equation}
and the collision operator $\mathcal{L}$ is strictly elliptic for any time $t \geq 0$ due to \eqref{direction temp}.
Similarly, if we fix some indices $\alpha\neq\beta$ and  take $\vphi(v)=v_\alpha v_\beta$, we can obtain that 
$$
\begin{aligned}
 \frac{\ud}{\ud t}\int_{\R^{d}} f(t,v)v_\alpha v_\beta\ud v
 =&2\int_{\R^{2d}} \big(-(v-w)_\alpha(v-w)_\beta\big)f(t,w)f(t,v)\ud w\ud v\\&-2(d-1)\int_{\R^{2d}}(v-w)_\alpha(v-w)_\beta f(t,w)f(t,v)\ud w\ud v\\
=&-4d\int_{\R^{d}} f(t,v)v_\alpha v_\beta\ud v,
\end{aligned}
$$
with initial condition $\int_{\R^d}f_0v_\alpha v_\beta=0,\mbox{ for } \alpha\neq \beta$.  This implies that for any $t\geq 0$,
\begin{equation}\label{cross conservation}
\me_{\alpha \beta}(t)=\int_{\R^d}f(t,v)v_\alpha v_\beta=0,\mbox{ for } \alpha\neq \beta.  
\end{equation}
Substitute the explicit form of $\me_{\alpha\beta}$ into  \eqref{bara} to get
$$
\Bar{a}_{\alpha\beta}=d\delta_{\alpha \beta}+(|v|^2\delta_{\alpha \beta}-v_\alpha v_\beta)-\delta_{\alpha \beta}(1+D_{\alpha\alpha}e^{-4dt}),
$$
and the simplified formulation of Maxweillian-Landau equation
\begin{equation}\label{simplified}
    \partial_t f=\ml f=\big( (d-1)\Delta+|v|^2 \Delta-v_\alpha v_\beta \partial_{\alpha \beta}-e^{-4dt}D_{\alpha \alpha}\partial_{\alpha \alpha}+d(d-1)\big)f.
\end{equation}
We can further decompose the collision operator into two parts as in \cite[Section 4]{villani1998spatially}, which are both independent of the orthonormal basis, say
\begin{equation*}
\ml f=\ml_1 f+\ml_2 f,
\end{equation*}
where the linear Fokker-Planck-like part reads as
\begin{equation*}
    \ml_1 f=(d-1)\Delta f+(d-1)\nabla\cdot ( fv )-e^{-4dt}D_{\alpha \alpha}\partial_{\alpha \alpha}f,
\end{equation*}
and the Laplace-Beltrami type operator part becomes
$$
\ml_2 f=(|v|^2\delta_{\alpha \beta}-v_\alpha v_\beta)\p_{\alpha\beta} f-(d-1)v \cdot \nabla f.$$

\subsection{Parabolic Maximum Principle}

As explained in previous texts, the key technique which we shall use in the main logarithmic estimates is the well-known parabolic maximum principle. Since the collision operator in the form of \eqref{simplified} is strictly elliptic, one may expect that $\partial_t-\ml$ satisfies the parabolic maximum principle as the classical heat equation. Our main logarithmic estimates rely on this property. We now give a detailed discussion of the principle in our collision operator setting.
\begin{prop}[Maximum Principle]\label{maximumprinciple}
    Suppose that $u \in \C^1([0,T],\C^2(\R^d))$ satisfies
    \begin{equation*}
        (\partial_t-\ml)u=(\partial_t-(d-1)\Delta-|v|^2\Delta+v_\alpha v_\beta \partial_{\alpha \beta}+e^{-4dt}D_{\alpha \alpha}\partial_{\alpha \alpha}-d(d-1))u \leq 0
    \end{equation*}
    with the initial condition
    \begin{equation*}
        u(0,\cdot) \leq 0.
    \end{equation*}
    Assume further that $u$ satisfies the a priori exponential growth condition with some constants $A,M>0$,
    \begin{equation}\label{aprioriexp}
        u(t,v) \leq Ae^{M|v|^2}
    \end{equation}
    for $t \in [0,T]$ and $v \in \R^d$. Then $u \leq 0$ holds for any $t \in [0,T]$.
\end{prop}

\begin{proof}
    We first assume that $$T<\min\left\{\frac{1}{2d(d-1)},\frac{1}{8dM}\right\},$$ which then follows that there exists some $\eps$ that  $$0<\eps<\min\left\{\frac{1}{2d(d-1)},\frac{1}{8dM}\right\}-T.$$ Our idea is similar to the proof of the maximum principle of the heat equation \cite[Section 2.3]{evans2022partial}, say firstly proving the result for a short time horizon $[0,T]$, and then proceeding from time $T$ to prove the result in $[T, 2T]$ and so on.
    
    We firstly need to change the variable to make the constant term in the parabolic operator to be 
    non-negative. Hence we set $w=e^{-(d^2-d+1)t}u$ to find that
    \begin{equation*}
        (\partial_t-\ml)w=e^{-(d^2-d+1)t}(\partial_t-\ml)u-(d^2-d+1)e^{-(d^2-d+1)t}u,
    \end{equation*}
which satisfies
    \begin{equation*}
    \begin{aligned}
  (\partial_t-\Tilde{\ml})w:=&(\partial_t-(d-1)\Delta-|v|^2\Delta+v_\alpha v_\beta \partial_{\alpha \beta}+e^{-4dt}D_{\alpha \alpha}\partial_{\alpha \alpha}+1)w \\=&(\p_t-\Bar{a}_{\alpha\beta}\partial_{\alpha\beta}+1)w=e^{-(d^2-d+1)t}(\partial_t-\ml)u
  \leq 0.     
    \end{aligned}
    \end{equation*}
Next, we construct a radially symmetric auxiliary function $\phi$ which satisfies $(\p_t-\Tilde{\ml})\phi\geq0$,  namely,
    \begin{equation*}
        \phi(t,v)= \mu e^{\beta(t)}e^{\delta(t) \frac{|v|^2}{2}}, 
    \end{equation*}
    where $\beta(t)$ and $\delta(t)$ are some functions to be determined, and $\mu$ is any positive number. Notice that the Laplace-Beltrami type part does not play a role on the radially symmetric function $\phi$,  i.e. $\ml_2\phi=0$, then it holds that
    \begin{align*}
        &(\p_t-\Tilde{\ml})\phi\\=&(\partial_t-(d-1)\Delta-(d-1)v\cdot \nabla +e^{-4dt}D_{\alpha \alpha}\partial_{\alpha \alpha}+1)\phi\\
        =&\Big(\beta^\prime+ \delta^\prime \frac{|v|^2}{2}-d(d-1)\delta-(d-1)\delta^2|v|^2-(d-1)\delta |v|^2+e^{-4dt}\delta^2 D_{\alpha \alpha}v_\alpha^2+1\Big)\phi\\
        \geq&\Big((\beta^\prime-d(d-1)\delta+1)+|v|^2\big(\frac{\delta^\prime}{2}-2(d-1)\delta^2-(d-1)\delta \big)\Big)\phi,
\end{align*}
since
$$
|D_{\alpha \alpha}v_\alpha^2|\leq (\sup|D_{\alpha \alpha}|)|v|^2\leq  (d-1-\eta)|v|^2.
$$
To satisfy $(\p_t-\Tilde{\ml})\phi\geq0$, it is sufficient to choose
    \begin{equation*}
        \beta^\prime \geq d(d-1)\delta-1, \quad \delta^\prime \geq 4(d-1)\delta^2+2(d-1)\delta.
    \end{equation*}
This can be verified by choosing $$\delta(t)=\frac{1}{4d(T+\eps-t)}\quad\text{and}\quad\beta(t)=\frac{(d-1)t}{4\eps}$$ thanks to the assumption on $T$. Now we summarize that $w-\phi$ satisfies
    \begin{equation*}
        (\partial_t-\Tilde{\ml})(w-\phi) \leq 0
    \end{equation*}
with $w(0)-\phi(0) \leq u(0)\leq 0$. Since $T+\eps<\frac{1}{8dM}$, for large enough $R>0$, in the region $|v|\geq R$  we have 
    \begin{equation*}
        w(t,v)-\phi(t,v) \leq e^{-(d^2-d+1)t}Ae^{MR^2}- \mu e^\frac{(d-1)t}{4\eps} e^\frac{R^2}{8d(T+\eps-t)} \leq 0.
    \end{equation*}
We prove by contradiction that if $w-\phi$ had its positive maximum in $[0,T] \times B_R$, say at $(t_0,v_0)$, then following conditions hold:
  $$
  \partial_t (w-\phi)(t_0,v_0) \geq 0,\quad \bar a_{\alpha\beta}\p_{\alpha\beta}(w-\phi)(t_0,v_0)\leq 0,\text{ and } (w-\phi)(t_0,v_0)\geq 0;
  $$
namely  $$(\partial_t-\Tilde{\ml})(w-\phi)=(\partial_t-\bar a_{\alpha\beta}\p_{\alpha\beta}+1)(w-\phi) \geq 0,$$
which contradicts with that $(\partial_t-\Tilde{\ml})(w-\phi) \leq 0$. Hence for any $\mu>0$, we have for sufficiently large $R$ that 
    $
        w(t,v) \leq \phi(t,v)
 $
    holds in $[0,T] \times \R^d$. Fix $w$ and let $\mu \rightarrow 0$, we have $w(t,v) \leq 0$ and thus $u(t,v) \leq 0$ for $t\in[0,T]$.


    Finally, for general $T>0$, we divide $[0,T]$ into many time intervals such that the length of each interval satisfies the assumption, and proceed as above.
\end{proof}

In the similar spirit of  Proposition \ref{maximumprinciple}, we can derive the Gaussian lower bound by the comparison principle, which is useful for further regularity estimates. This has been done in \cite[Theorem 3]{villani1998spatially}, once we notice that the initial growth condition \eqref{inigradient} implies directly that the initial density has Gaussian lower bound, say
$f_0(v) \geq C^{-1}\exp(-C|v|^2)$
for some large constant $C$.

\begin{prop}[Gaussian Lower Bound]\label{lowergauss}
    Assume that the solution of the Landau-Maxwellian equation $f \in \C^1([0,T],\C^2(\R^d))$ with its initial data
    satisfying the logarithmic gradient growth estimate \eqref{inigradient}
    \begin{equation*}
        |\nabla \log f_0| \leq C_1(1+|v|),
    \end{equation*}
    then for any $t \in [0,T]$ we have
    \begin{equation*}
        f(t,v) \geq C_2 \exp\Big(-C_2'(t) \frac{|v|^2}{2}\Big)
    \end{equation*}
    for some  constant $C_2$ and $C_2(t)>1$ with $C_2'(t) \rightarrow 1$ as $t \rightarrow \infty$.
\end{prop}

\section{Logarithmic Estimates}\label{estimate}

In this section we shall give the logarithmic gradient estimate and logarithmic Hessian estimate as in \cite{feng2023quantitative}, which are necessary for evaluating the relative entropy. Our goal is to derive the linear growth of $|\nabla \log f|$ and quadratic growth of $|\nabla^2 \log f|$, once the same conditions are satisfied by the initial case. We shall apply the Bernstein method, say constructing some auxiliary functions and propagating them by the parabolic operator. Using the maximum principle Proposition \ref{maximumprinciple} and the initial condition, these functions are proved to be non-negative, which gives bounds on our desired quantities, namely Theorem \ref{loggradient} and Theorem \ref{loghessian}. Proofs are given respectively in the coming subsections. 
\subsection{Logarithmic Gradient Estimate}
    We will need the following two results of propagation, which are nothing but elementary calculations.
    \begin{lemma}\label{evo1}
         $(\partial_t-\ml)\frac{|\nabla f|^2}{f} \leq 4d\frac{|\nabla f|^2}{f}$.
    \end{lemma}
\begin{lemma}\label{evo1log}
$(\partial_t-\ml) f\log f \leq -\eta \frac{|\nabla f|^2}{f}+d(d-1)f,$ where $\eta$ is the small positive constant given in \eqref{eta}.
\end{lemma}
For simplicity of the proof, we first prove our main estimate assuming the above two lemmas. The technical proofs of Lemma \ref{evo1} and Lemma \ref{evo1log} will be postponed to the end of this section.

Now we construct the auxiliary function to be propagated as
\begin{equation*} 
    F(t,v)=\frac{|\nabla f|^2}{f}+Cf\log f-d(d-1)Ctf-C^\prime f.
\end{equation*}
From the two lemmas above we have
\begin{equation*}
    (\partial_t-\ml)F \leq 4d\frac{|\nabla f|^2}{f}-C\eta \frac{|\nabla f|^2}{f}+d(d-1)Cf-d(d-1)Cf \leq 0,
\end{equation*}
for any $C\geq\frac{4d}{\eta}$.

The formula at $t=0$ simply reads
\begin{equation*}
    F(0,v)=\frac{|\nabla f_0|^2}{f_0}+Cf_0\log f_0-C^\prime f_0.
\end{equation*}
By the initial condition \eqref{inigauss} and \eqref{inigradient}, we choose $C$ large enough and then $C^\prime$ large enough to make $F(0,v) \leq 0$. Notice that $f$ is bounded from \cite{villani1998spatially} and that $f$ has Gaussian lower bound from Proposition \ref{lowergauss}. Hence in order to check that the auxiliary function $F$ satisfies the {\em a priori} exponential growth condition \eqref{aprioriexp}, we only need to prove the boundedness of $\nabla f$. This is obtained from \cite[Theorem 2.1]{guillen2023landau}, which is the space-homogeneous case of the previous result Henderson-Snelson-Tarfulea \cite[Theorem 1.2]{henderson2020local}. Hence by the parabolic maximum principle Theorem \ref{maximumprinciple}, we have $F \leq 0$ for any $t$, or
\begin{equation}\label{gradient logf}
    |\nabla \log f|^2 \leq Cd(d-1)t-C\log f+C^\prime.
\end{equation}
We also note that from Proposition \ref{lowergauss} we have
$$
    \log f \geq -M_2-C_2(t)|v|^2.
$$
This will give our main estimate Theorem \ref{loggradient}.

Now we present the proofs of the two technical lemmas.
\begin{proof}[Proof of Lemma \ref{evo1}]
We first list some elementary calculations.
\begin{claim}[Elementary Calculations I]\label{cal1}
 $$
        \partial_t \frac{|\nabla f|^2}{f}=-\frac{|\nabla f|^2}{f^2} \partial_t f+\frac{2}{f} \nabla f \cdot \nabla \partial_t f,
 $$    
$$\partial_\alpha \frac{|\nabla f|^2}{f}=-\frac{|\nabla f|^2}{f^2} \partial_\alpha f+\frac{2}{f} \nabla f \cdot \nabla \partial_\alpha f, $$   \begin{align*}&
\partial_{\alpha \beta} \frac{|\nabla f|^2}{f}=-\frac{|\nabla f|^2}{f^2} \partial_{\alpha \beta} f+\frac{2}{f} \nabla f \cdot \nabla \partial_{\alpha \beta} f\\
&\quad \quad +\Big(\frac{2}{f}\nabla \partial_\alpha f \cdot \nabla \partial_\beta f-\frac{2}{f^2} \partial_\beta f \nabla f \cdot \nabla \partial_\alpha f-\frac{2}{f^2} \partial_\alpha f \nabla f \cdot \nabla \partial_\beta f+\frac{2}{f^3}\partial_\alpha f \partial_\beta f |\nabla f|^2\Big).
    \end{align*}
\end{claim}
By Claim \ref{cal1} we propagate that
\begin{equation}\label{Proof terms}
\begin{aligned}
    &\Big(\partial_t-(d-1)\Delta-|v|^2 \Delta+v_\alpha v_\beta \partial_{\alpha \beta}+e^{-4dt}D_{\alpha \alpha}\partial_{\alpha \alpha}-d(d-1)\Big) \frac{|\nabla f|^2}{f}\\
    =&-\frac{|\nabla f|^2}{f^2}\Big(\partial_t f-(d-1)\Delta f-|v|^2 \Delta f+v_\alpha v_\beta \partial_{\alpha \beta}f+e^{-4dt}D_{\alpha \alpha}\partial_{\alpha \alpha}f+d(d-1)f\Big)\\
    &+\frac{2}{f}\nabla f \cdot \nabla\Big(\partial_t f-(d-1)\Delta f-|v|^2 \Delta f+v_\alpha v_\beta \partial_{\alpha \beta}f+e^{-4dt}D_{\alpha \alpha}\partial_{\alpha \alpha}f\Big)\\&-\frac{2}{f}\Big(d-1-e^{-4dt}D_{\alpha \alpha}\Big)\Big|\partial_{\alpha \beta}f-\frac{\partial_\alpha f \partial_\beta f}{f}\Big|^2\\
    &+\frac{4}{f}(v \cdot \nabla f)\Delta f-\frac{4}{f}v_\beta \partial_\alpha f \partial_{\alpha \beta}f-\frac{2}{f}|v|^2\Big|\partial_{\alpha \beta}f-\frac{\partial_\alpha f \partial_\beta f}{f}\Big|^2\\&+\frac{2}{f}v_\alpha v_\beta \Big(\nabla \partial_\alpha f \cdot \nabla \partial_\beta f-\frac{2}{f}\partial_\beta f \nabla f \cdot \nabla \partial_\alpha f+\frac{|\nabla f|^2}{f^2} \partial_\alpha f \partial_\beta f\Big)\\
=&\textbf{Term}_1+\textbf{Term}_2+\textbf{Term}_3+\textbf{Term}_4+\textbf{Term}_5.
\end{aligned}
\end{equation}

We then analyze these terms respectively.
$\textbf{Term}_1+\textbf{Term}_2$ can be simplified directly by the evolution equation of $f$, which shows that the sum of them vanishes;
$\textbf{Term}_3$ is obviously non-positive since $D_{\alpha\alpha}  \leq d-1$; The last two terms in $\textbf{Term}_4+\textbf{Term}_5$ can be changed into negative sum of square terms. 

We first expand it as
\begin{align*}
   \frac{2}{f}\sum_{\alpha,\beta,\gamma}\Big(&-v_{\gamma}^2( \p_{\alpha\beta} f)^2+2v_{\gamma}^2 \p_{\alpha\beta} f\frac{ \p_{\alpha}  f  \p_{\beta}  f}{f}-v_{\gamma}^2\frac{( \p_{\alpha}  f)^2( \p_{\beta}  f)^2}{f^2}\\
   &+v_{\alpha}v_{\beta}\p_{\alpha\gamma}f\p_{\beta\gamma}f-2v_{\alpha}v_{\beta}\frac{\p_{\alpha\gamma}f \p_{\beta} f\partial_{\gamma}f}{f}+v_{\alpha}v_{\beta}\frac{ \p_{\alpha} f \p_{\beta} f(\partial_{\gamma}f)^2}{f^2}\Big),
\end{align*}
then divide it further into three parts:
$$
    -v_{\gamma}^2( \p_{\alpha\beta} f)^2+v_{\alpha}v_{\beta}\p_{\alpha\gamma}f\p_{\beta\gamma}f
    =-\frac{1}{2}(v_{\alpha}\p_{\beta\gamma}f-v_{\beta}\p_{\alpha\gamma}f)^2,$$ $$
    -v_{\gamma}^2\frac{( \p_{\alpha}  f)^2( \p_{\beta}  f)^2}{f^2}+v_{\alpha}v_{\beta}\frac{ \p_{\alpha} f \p_{\beta} f(\partial_{\gamma}f)^2}{f^2}=-\frac{1}{2}\Big(v_{\alpha}\frac{ \p_{\beta} f\partial_{\gamma}f}{f}-v_{\beta}\frac{ \p_{\alpha} f\partial_{\gamma}f}{f}\Big)^2,$$
    \begin{align*}
    &2v_{\gamma}^2 \p_{\alpha\beta} f\frac{ \p_{\alpha}  f  \p_{\beta}  f}{f}-2v_{\alpha}v_{\beta}\frac{\p_{\alpha\gamma}f \p_{\beta} f\partial_{\gamma}f}{f}=\Big(v_{\alpha}\p_{\beta\gamma}f-v_{\beta}\p_{\alpha\gamma}f\Big)\Big(v_{\alpha}\frac{ \p_{\beta} f\partial_{\gamma}f}{f}-v_{\beta}\frac{ \p_{\alpha} f\partial_{\gamma}f}{f}\Big).
\end{align*}
Hence we are able to rewrite $\textbf{Term}_4+\textbf{Term}_5$ into a more compact form as
\begin{equation}\label{negative}
\begin{aligned}
    &\frac{4}{f}(v_{\alpha} \p_{\beta\beta} f-v_{\beta} \p_{\alpha\beta} f) \p_{\alpha} f-\frac{1}{f}\Big(v_{\alpha}\p_{\beta\gamma}f-v_{\beta}\p_{\alpha\gamma}f-v_{\alpha}\frac{ \p_{\beta} f\partial_{\gamma}f}{f}+v_{\beta}\frac{ \p_{\alpha} f\partial_{\gamma}f}{f}\Big)^2\\
=&-\frac{1}{f}\sum_{\gamma\neq \beta}\sum_{\alpha}\Big(v_{\alpha}\p_{\beta\gamma}f-v_{\beta}\p_{\alpha\gamma}f-v_{\alpha}\frac{ \p_{\beta} f\partial_{\gamma}f}{f}+v_{\beta}\frac{ \p_{\alpha} f\partial_{\gamma}f}{f}\Big)^2
\\&+\frac{4}{f}\sum_{\alpha,\beta
}(v_{\alpha} \p_{\beta\beta} f-v_{\beta} \p_{\alpha\beta} f) \p_{\alpha} f\\&-\frac{1}{f}\sum_{ \alpha,\beta}\Big(v_{\alpha}\partial_{\beta \beta}f-v_{\beta}\partial_{\alpha \beta}f-v_{\alpha}\frac{ \p_{\beta} f\partial_\beta f}{f}+v_{\beta}\frac{ \p_{\alpha} f\partial_\beta f}{f}\Big)^2.
\end{aligned}
\end{equation}
Notice that the cancelling property holds
\begin{equation*}
\sum_{\alpha,\beta
}\Big(v_{\alpha}\frac{ \p_{\beta} f \p_{\beta} f}{f}-v_{\beta}\frac{ \p_{\alpha} f \p_{\beta} f}{f}\Big) \p_{\alpha} f=\frac{|\nabla f|^2}{f}v \cdot \nabla f-\frac{|\nabla f|^2}{f}v \cdot \nabla f=0.
\end{equation*}
The last two lines  of \eqref{negative}  are then equivalent to 
\begin{equation*}
    -\frac{1}{f}\sum_{\alpha,\beta
}\Big(v_{\alpha} \p_{\beta\beta} f-v_{\beta} \p_{\alpha\beta} f-v_{\alpha}\frac{ \p_{\beta} f \p_{\beta} f}{f}+v_{\beta}\frac{ \p_{\alpha} f \p_{\beta} f}{f}-2 \p_{\alpha} f\Big)^2+4d\frac{|\nabla f|^2}{f}.
\end{equation*}
This concludes our desired propagation estimate.
\end{proof}

\begin{proof}[Proof of Lemma \ref{evo1log}]
    We also list some other elementary calculations first.
    \begin{claim}[Elementary Calculations II]\label{cal2}
    \begin{align*}
        \partial_t (f\log f) &=(1+\log f)\p_t f.\\
        \p_\alpha (f\log f) &=(1+\log f)\p_\alpha f.\\
    \p_{\alpha\beta} (f\log f)&=(1+\log f)\p_{\alpha\beta}f+\frac{\p_\alpha f\p_\beta f}{f}.
    \end{align*}
\end{claim}
By claims above we propagate that
\begin{align*}
    &\Big(\partial_t-(d-1)\Delta-|v|^2 \Delta+v_{\alpha} v_{\beta}  \p_{\alpha\beta} +e^{-4dt}D_{\alpha\alpha}  \p_{\alpha\alpha} -d(d-1)\Big) f\log f\\
    =&\Big(1+\log f\Big)\Big(\partial_t f-(d-1)\Delta f-|v|^2\Delta f+v_{\alpha}v_{\beta} \p_{\alpha\beta} f+e^{-4dt}D_{\alpha\alpha}  \p_{\alpha\alpha} f-d(d-1)f\Big)\\
    &+d(d-1)f
    -\Big(d-1-e^{-4dt}D_{\alpha\alpha} \Big)\frac{( \p_{\alpha} f)^2}{f}-\Big(|v|^2\frac{|\nabla f|^2}{f}-v_{\alpha}v_{\beta}\frac{ \p_{\alpha} f \p_{\beta} f}{f}\Big)\\
    \leq& -\eta \frac{|\nabla f|^2}{f}+d(d-1)f,
\end{align*}
where the last term before the inequality coincides with 
\begin{equation}\label{product identity}
|v|^2\frac{|\nabla f|^2}{f}-\frac{(v\cdot \nabla f)^2}{f}=\frac{1}{2f}\big(v_\alpha\p_\beta f-v_\beta\p_\alpha f\big)^2.
\end{equation}
This will give our desired estimate.
\end{proof}

\subsection{Logarithmic Hessian Estimate}

    It suffices to bound $\frac{\nabla^2 f}{f}$ by some quadratic function as in \cite{feng2023quantitative}. We will also need the following two results of propagation as in the previous proof of Theorem \ref{loggradient}.
    \begin{lemma}\label{evo2}
        $(\partial_t-\ml)\frac{|\nabla^2 f|^2}{f} \leq 16d\frac{|\nabla^2 f|^2}{f}$.
    \end{lemma}
    \begin{lemma}\label{evo2log}
        \begin{align*}
            (\partial_t-\ml)f(\log f)^2 =&2d(d-1)f\log f-(d-1-e^{-4dt}D_{\alpha\alpha} )\frac{2}{f}(1+\log f)( \p_{\alpha} f)^2\\
            &-\frac{2}{f}(1+\log f)(v_{\alpha} \p_{\beta} f-v_{\beta} \p_{\alpha} f)^2.
        \end{align*}
    \end{lemma}
For simplicity of the proof, we first prove our main estimate assuming the above two lemmas. The technical proofs of Lemma \ref{evo2} and Lemma \ref{evo2log} will be postponed to the end of this subsection.
Recall in the proof of Lemma \ref{evo1}, if we keep $\textbf{Term}_3$ in \eqref{Proof terms} then by the Cauchy-Schwarz inequality we get 
\begin{align*}
    (\partial_t-\ml) \frac{|\nabla f|^2}{f} &\leq 4d\frac{|\nabla f|^2}{f}-\frac{2}{f}(d-1-e^{-4dt}D_{\alpha\alpha} )\Big| \p_{\alpha\beta} f-\frac{ \p_{\alpha} f \p_{\beta} f}{f}\Big|^2\\
    &\leq 4d\frac{|\nabla f|^2}{f}-\eta \frac{|\nabla^2 f|^2}{f}+(d-1-e^{-4dt}D_{\alpha\alpha} )\frac{2}{f}\frac{|\nabla f|^2}{f^2}( \p_{\alpha} f)^2.
\end{align*}
Also the identity \eqref{product identity} in the proof of Lemma \ref{evo1log} implies
\begin{equation}\label{flogf}
\begin{aligned}
    (\partial_t-\ml) f\log f
    \leq d(d-1)f-(d-1-e^{-4dt}D_{\alpha\alpha} )\frac{( \p_{\alpha} f)^2}{f}-\frac{1}{2f}(v_{\alpha} \p_{\beta} f-v_{\beta} \p_{\alpha} f)^2.
\end{aligned}
\end{equation}
Notice that for large enough $A \geq 4+4\sup( \log f)_+$, which exists since the solution of the Landau-Maxwellian equation is bounded, as explained in Remark \ref{classical solution}, we have
\begin{align*}
    &(\partial_t-\ml)\big(-f(\log f)^2+(2d(d-1)t+A)f\log f\big)\\
    \leq &-2d(d-1)f\log f+(d-1-e^{-4dt}D_{\alpha\alpha} )\frac{2}{f}( \p_{\alpha} f)^2(1+\log f)\\&+\frac{2+2\log f}{f}(v_{\alpha} \p_{\beta} f-v_{\beta} \p_{\alpha} f)^2+2d(d-1)f\log f+(2d^2(d-1)^2t+d(d-1)A)f\\&-(d-1-e^{-4dt}D_{\alpha \alpha})\frac{2}{f}(\partial_\alpha f)^2(d(d-1)t+\frac{A}{2})-\frac{2d(d-1)t+A}{2f}(v_{\alpha} \p_{\beta} f-v_{\beta} \p_{\alpha} f)^2\\
    \leq &(d-1-e^{-4dt}D_{\alpha\alpha} )\frac{2}{f}( \p_{\alpha} f)^2\log f+(2d^2(d-1)^2t+d(d-1)A)f.
\end{align*}
Also recall the inequality \eqref{gradient logf} we obtain in the proof of Theorem \ref{loggradient} that
\begin{equation}\label{logft}
   \frac{1}{C} \frac{|\nabla f|^2}{f^2}+\log f \leq d(d-1)t+\frac{C^\prime}{C}, 
\end{equation}
for any sufficiently large $C$ and then large enough $C^\prime$. Hence by using $d-1-e^{-4dt}D_{\alpha\alpha}\geq \eta$ again, it holds that
\begin{align*}
    &(\partial_t-\ml)\Big(-f(\log f)^2+(2d(d-1)t+A)f\log f+\frac{1}{C}\frac{|\nabla f|^2}{f}\Big)\\
    \leq &(d-1-e^{-4dt}D_{\alpha\alpha})\big(\frac{4d}{C\eta}+2d(d-1)t+\frac{2C^\prime}{C}\big)\frac{(\p_\alpha f)^2}{f}\\&-\frac{\eta}{C}\frac{|\nabla^2f|^2}{f}+(2d^2(d-1)^2t+d(d-1)A)f.
\end{align*}
Recall \eqref{flogf} to get
$$
\begin{aligned}
&(\p_t-\ml)\Big(\big(\frac{4d}{C\eta}+2d(d-1)t+\frac{2C^\prime}{C}\big) f\log f \Big)\\
    \leq &-(d-1-e^{-4dt}D_{\alpha\alpha} )\Big(\frac{4d}{C\eta}+2d(d-1)t+\frac{2C^\prime}{C}\Big)\frac{( \p_{\alpha} f)^2}{f}\\&+
    \Big(\frac{4d^2(d-1)}{C\eta}+2d^2(d-1)^2t+\frac{2d(d-1)C^\prime}{C}\Big)f+2d(d-1)f\log f.
\end{aligned}    
$$
Then combining with Lemma \ref{evo2}, we evaluate the following quantity to obtain
$$
\begin{aligned}
&(\p_t-\ml)\Big[-f(\log f)^2+\big(2d(d-1)t+A\big)f\log f+\frac{1}{C}\frac{|\nabla f|^2}{f}+\big(\frac{4d}{C\eta}+2d(d-1)t+\frac{2C^\prime}{C}\big)\\
&\qquad \qquad f\log f+\frac{\eta}{16dC}\frac{|\nabla^2 f|^2}{f} \Big]\\
    &\leq \Big(\frac{4d^2(d-1)}{C\eta}+4d^2(d-1)^2t+d(d-1)A+\frac{2d(d-1)C^\prime}{C}\Big)f+2d(d-1)f\log f\\
&\leq \Big(6d^2(d-1)^2t+\frac{4d^2(d-1)}{C\eta}+d(d-1)A+\frac{4d(d-1)C^\prime}{C}\Big)f,
\end{aligned}
$$
where we use that $2d(d-1)f\log f\leq 2d^2(d-1)^2tf+\frac{2d(d-1)C^\prime}{C}f$ by \eqref{logft}.
In terms of the computation above, we construct a new auxiliary function by adding terms compensating the right-hand side above to let $
    (\partial_t  -\ml)F \leq 0$, which reads as
\begin{align*}
    F(t,v)=&-f(\log f)^2+\big(2d(d-1)t+A\big)f\log f+\frac{1}{C}\frac{|\nabla f|^2}{f}\\
    &+\big(\frac{4d}{C\eta}+2d(d-1)t+\frac{2C^\prime}{C} \big)f\log f+\frac{\eta}{16dC}\frac{|\nabla^2 f|^2}{f}\\&-\Big(3d^2(d-1)^2t^2+d(d-1)At+\frac{4d^2(d-1)}{C\eta}t+\frac{4d(d-1)C^\prime}{C}t+C^{\prime\prime}\Big)f\\=&\frac{\eta}{16dC}\frac{|\nabla^2 f|^2}{f}-f(\log f)^2+\Big(4d(d-1)t+A+\frac{4d}{C\eta}+\frac{2C^\prime}{C}\Big)f\log f+\frac{1}{C}\frac{|\nabla f|^2}{f}\\
    &-\Big(3d^2(d-1)^2t^2+d(d-1)At+\frac{4d^2(d-1)}{C\eta}t+\frac{4d(d-1)C^\prime}{C}t+C^{\prime\prime}\Big)f.
\end{align*}
The initial case holds 
\begin{equation*}
    F(0,v)=\frac{\eta}{16dC}\frac{|\nabla^2 f_0|^2}{f_0}-f_0(\log f_0)^2+\Big(A+\frac{4d}{C\eta}+\frac{2C^\prime}{C}\Big)f_0\log f_0+\frac{1}{C}\frac{|\nabla f_0|^2}{f_0}-C^{\prime\prime}f_0.
\end{equation*}
By the initial condition \eqref{inigradient} \eqref{inihessian}, we choose one-by-one $C, C^\prime, C^{\prime\prime}$ large enough to make $F(0,v) \leq 0$. Also the {\em a priori} exponential growth condition \eqref{aprioriexp} is satisfied similar to the logarithmic gradient estimate, with the boundedness of $\nabla^2 f$ is also obtained from \cite[Theorem 1.2]{henderson2020local} and \cite[Theorem 2.1]{guillen2023landau}. Now we apply the parabolic maximum principle Theorem \ref{maximumprinciple}, which implies $F \leq 0$ for any $t$, or
\begin{equation*}
    \frac{|\nabla^2 f|^2}{f^2} \leq C\Big((\log f)^2-(1+t)\log f+|\nabla \log f|^2+(1+t+t^2)\Big)
\end{equation*}
for some $C=C(d,\eta, A)$. We also note that from Proposition \ref{lowergauss} and \cite{villani1998spatially} we have
\begin{equation*}
    -M_2-C_2(t)|v|^2 \leq \log f \leq M_1.
\end{equation*}
This will give our main estimate Theorem \ref{loghessian}.

Now we present the proofs of the two technical lemmas.
\begin{proof}[Proof of Lemma \ref{evo2}]
We again start with some elementary calculations. 
\begin{claim}[Elementary Calculations III]\label{cal3}
$$
        \partial_t \frac{|\nabla^2 f|^2}{f}=-\frac{|\nabla^2f|^2}{f^2}\partial_t f+\frac{2}{f}\nabla^2f:\nabla^2\partial_t f.$$ $$
         \p_{\alpha}  \frac{|\nabla^2 f|^2}{f}=-\frac{|\nabla^2f|^2}{f^2} \p_{\alpha}  f+\frac{2}{f}\nabla^2f:\nabla^2 \p_{\alpha}  f.$$
\begin{align*}
\p_{\alpha\beta}  \frac{|\nabla^2 f|^2}{f}=&-\frac{|\nabla^2f|^2}{f^2} \p_{\alpha\beta}  f+\frac{2}{f}\nabla^2f:\nabla^2 \p_{\alpha\beta}  f\\
        &+\frac{2}{f^3}|\nabla^2f|^2 \p_{\alpha} f \p_{\beta} f+\frac{2}{f}\nabla^2 \p_{\alpha} f:\nabla^2 \p_{\beta} f\\
        &-\frac{2}{f^2}\nabla^2f:\nabla^2 \p_{\beta} f \p_{\alpha} f-\frac{2}{f^2}\nabla^2f:\nabla^2 \p_{\alpha} f \p_{\beta} f.
\end{align*}
\end{claim}
By Claim \ref{cal3} we propagate that
\begin{align*}
    &\Big(\partial_t-(d-1)\Delta-|v|^2\Delta+v_{\alpha}v_{\beta} \p_{\alpha\beta} +e^{-4dt}D_{\alpha\alpha}  \p_{\alpha\alpha} -d(d-1)\Big)\frac{|\nabla^2f|^2}{f}\\
    =&-\frac{|\nabla^2f|^2}{f^2}\Big(\partial_tf-(d-1)\Delta f-|v|^2\Delta f+v_{\alpha}v_{\beta} \p_{\alpha\beta} f+e^{-4dt}D_{\alpha\alpha}  \p_{\alpha\alpha} f+d(d-1)f\Big)\\
    &+\frac{2}{f}\nabla^2f:\nabla^2\Big(\partial_tf-(d-1)\Delta f-|v|^2\Delta f+v_{\alpha}v_{\beta} \p_{\alpha\beta} f+e^{-4dt}D_{\alpha\alpha}  \p_{\alpha\alpha} f\Big)\\
    &-\frac{2}{f}(d-1-e^{-4dt}D_{\alpha\alpha} )\Big|\p_{\alpha \beta \gamma}f-\frac{ \p_{\alpha} f\p_{\beta\gamma}f}{f}\Big|^2\\
    &+\frac{4}{f}\Big(|\Delta f|^2+2v_{\alpha} \p_{\alpha\beta} f\Delta \p_{\beta} f-|\nabla^2f|^2-2v_{\alpha}\p_{\beta\gamma}f\p_{\alpha \beta \gamma}f\Big)-\frac{2}{f}|v|^2\Big|\p_{\alpha \beta \gamma}f-\frac{ \p_{\alpha} f\p_{\beta\gamma}f}{f}\Big|^2\\&+\frac{2}{f}v_{\alpha}v_{\beta}\Big(\nabla^2 \p_{\alpha} f:\nabla^2 \p_{\beta} f-\frac{2}{f} \p_{\alpha} f\nabla^2f:\nabla^2 \p_{\beta} f+\frac{|\nabla^2f|^2}{f^2} \p_{\alpha} f \p_{\beta} f\Big)\\
=&\textbf{Term}_1+\textbf{Term}_2+\textbf{Term}_3+\textbf{Term}_4+\textbf{Term}_5.
\end{align*}
Similarly as what we did in the previous subsection, we can see the sum of first two terms $\textbf{Term}_1+\textbf{Term}_2$ vanishes; $\textbf{Term}_3$ is obviously non-positive since $D_{\alpha\alpha}  \leq d-1$.

To deal with $\textbf{Term}_4+\textbf{Term}_5$, we expand the sum of $\textbf{Term}_5$ and the second part of $\textbf{Term}_4$ together into sum of square as
\begin{equation}\label{negative2}
    -\frac{1}{f}\Big(v_{\alpha}\partial_{\beta \gamma \theta}f-v_{\beta}\partial_{\alpha \gamma \theta}f-v_{\alpha}\frac{ \p_{\beta} f\p_{\gamma \theta}f}{f}+v_{\beta}\frac{ \p_{\alpha} f\p_{\gamma \theta}f}{f}\Big)^2.
\end{equation}
 This again allows us to combine with the first part of $\textbf{Term}_4$, which satisfies
\begin{align*}
    &\frac{4}{f}(|\Delta f|^2+2v_{\alpha} \p_{\alpha\beta} f\Delta \p_{\beta} f-|\nabla^2f|^2-2v_{\alpha}\p_{\beta\gamma}f\p_{\alpha \beta \gamma}f)\\
    \leq &\frac{8}{f}\big(v_{\alpha}\partial_{\beta \beta \gamma}f-v_{\beta}f\p_{\alpha \beta \gamma}f\big)\p_{\alpha\gamma}f .
\end{align*}
Note also the cancelling property
\begin{equation*}
    \sum_{\alpha,\beta,\gamma}\Big(v_{\alpha}\frac{ \p_{\beta} f\p_{\beta\gamma}f}{f}-v_{\beta}\frac{ \p_{\alpha} f\p_{\beta\gamma}f}{f}\Big)\p_{\alpha\gamma}f=0.
\end{equation*}
Hence we extract terms with $\theta=\beta$ from \eqref{negative2}  to obtain
\begin{equation*}
\begin{aligned}
&\textbf{Term}_4+\textbf{Term}_5\\\leq &-\frac{1}{f}\sum_{\alpha,\gamma,\beta\neq\theta}\Big(v_{\alpha}\partial_{\beta \gamma \theta}f-v_{\beta}\partial_{\alpha \gamma \theta}f-v_{\alpha}\frac{ \p_{\beta} f\p_{\gamma \theta}f}{f}+v_{\beta}\frac{ \p_{\alpha} f\p_{\gamma \theta}f}{f}\Big)^2\\
    &-\frac{1}{f}\sum_{\alpha,\gamma,\beta}\Big(v_{\alpha}\partial_{\beta \beta \gamma}f-v_{\beta}\p_{\alpha \beta \gamma}f-v_{\alpha}\frac{ \p_{\beta} f\p_{\beta\gamma}f}{f}+v_{\beta}\frac{ \p_{\alpha} f\p_{\beta\gamma}f}{f}-4\p_{\alpha\gamma}f\Big)^2\\&+16d\frac{|\nabla^2 f|^2}{f}.
\end{aligned}
\end{equation*}
Summarizing all estimates above, we arrive at our desired propagation estimate.
\end{proof}

\begin{proof}[Proof of Lemma \ref{evo2log}]
    We also list some elementary calculations first as following claims.
    \begin{claim}[Elementary Calculations IV]\label{cal4}
    \begin{align*}
        \partial_t (f(\log f)^2) &=((\log f)^2+2\log f)\partial_t f,\\
         \p_{\alpha}  (f(\log f)^2) &=((\log f)^2+2\log f) \p_{\alpha}  f,\\
         \p_{\alpha\beta}  (f(\log f)^2) &=((\log f)^2+2\log f) \p_{\alpha\beta}  f+\frac{2}{f}(1+\log f) \p_{\alpha} f \p_{\beta} f.
    \end{align*}
\end{claim}
By Claim \ref{cal4} we propagate that
\begin{align*}
    &\Big(\partial_t-(d-1)\Delta-|v|^2 \Delta+v_\alpha v_\beta \partial_{\alpha\beta}+e^{-4dt}D_{\alpha\alpha}\partial_{\alpha\alpha}-d(d-1)\Big) \big(f(\log f)^2\big)\\
    =&((\log f)^2+2\log f)(\partial_t f-(d-1)\Delta f-|v|^2\Delta f+v_\alpha v_\beta\partial_{\alpha\beta}f\\&+e^{-4dt}D_{\alpha\alpha}\partial_{\alpha\alpha}f-d(d-1)f)-\frac{2}{f}(1+\log f)(v_\alpha\partial_\beta f-v_\beta\partial_\alpha f)^2\\
    &+2d(d-1)f\log f
    -(d-1-e^{-4dt}D_{\alpha\alpha})\frac{2}{f}(1+\log f)(\partial_\alpha f)^2\\
    =&2d(d-1)f\log f-(d-1-e^{-4dt}D_{\alpha\alpha} )\frac{2}{f}(1+\log f)( \p_{\alpha} f)^2\\&-\frac{2}{f}(1+\log f)(v_{\alpha} \p_{\beta} f-v_{\beta} \p_{\alpha} f)^2.
\end{align*}
This will give our desired estimate.
\end{proof}

\section{Proof of the Law of Large Numbers}\label{poc}

This section is devoted to the proof of the important Law of Large Numbers estimate Theorem \ref{LLN}.  The main idea is to take advantage of the some important quantities of Landau-Maxwellian equation, namely the conservation laws \eqref{conservation} with \eqref{cross conservation}, and the explicit form of the directional temperature \eqref{direction temp}: for $t\in[0,T]$,
$$\int_{\R^d}v_\alpha f(t,v)\ud v=0 , \qquad\int_{\R^d}|v|^2f(t,v)\ud v=d,$$
$$
\int_{\R^d}v_\alpha^2f(t,v)\ud v=\me_\alpha(t),\qquad \int_{\R^d}v_\alpha v_\beta f(t,v)\ud v=0,\mbox{ for } \alpha\neq \beta.
$$

We expand the matrix term into entries by using the equations above as
$$
\begin{aligned}
  &a\ast f(v^{i})-a(v^{i}-v^{j}) \\
 =& \Id\big(\int_{\R^3}|v^i-z|^2f(z)\ud z-|v^i-v^j|^2\big)\\&-\big(\int_{\R^3}(v^i-z)\otimes (v^i-z)f(z)\ud z-(v^i-v^j)\otimes (v^i-v^j)\big)
\\
=&\Id\big(d+2v^i\cdot v^j-|v^j|^2\big)-\big(\int_{\R^d}z\otimes zf(z)\ud z+v^i\otimes v^j+v^j\otimes v^i-v^j\otimes v^j\big),\\
\end{aligned}
$$
which can be also written in terms of  components as
\begin{equation}\label{component a}
a\ast f(v^{i})-a(v^{i}-v^{j})=\big(d+2v^i\cdot v^j-|v^j|^2\big)\delta_{\alpha\beta}-\big(\mathcal{E_\alpha}(t)\delta_{\alpha\beta}+v^i_\alpha v^j_\beta+v^j_\alpha v^i_\beta-v^j_\alpha v^j_\beta\big). 
\end{equation}

The quantity we need to estimate now turns into
$$
\begin{aligned}
&\int_{\R^{dN}} F_N\frac{1}{N}\sum_{i=1}^{N}\Big|a\ast f(v^{i})-\frac{1}{N}\sum_{j=1}^{N}a(v^{i}-v^{j})
\Big|^2\ud V \\
= & \sum_{\alpha,\beta=1}^d  \int_{\R^{dN}} F_N \frac{1}{N^3}\sum_{i,j,k=1}^{N}\Big[a\ast f(v^{i})-a(v^{i}-v^{j})
\Big]_{\alpha\beta}\Big[a\ast f(v^{i})-a(v^{i}-v^{k})
\Big]_{\alpha\beta} \ud V\\
= & \sum_{\alpha,\beta=1}^d  \int_{\R^{dN}} F_N \frac{1}{N^3}\sum_{i\neq j\neq k}\Big[a\ast f(v^{i})-a(v^{i}-v^{j})
\Big]_{\alpha\beta}\Big[a\ast f(v^{i})-a(v^{i}-v^{k})
\Big]_{\alpha\beta} \ud V\\&+\sum_{\alpha,\beta=1}^d  \int_{\R^{dN}} F_N \frac{1}{N^3}\sum_{\mathcal{S} }\Big[a\ast f(v^{i})-a(v^{i}-v^{j})
\Big]_{\alpha\beta}\Big[a\ast f(v^{i})-a(v^{i}-v^{k})
\Big]_{\alpha\beta} \ud V,
 \end{aligned}
$$
where the index set $\mathcal{S}$ contains those triple indices where at least two of the indices are the same. Now the last term is bounded using 4-th order estimate in Proposition \ref{moment estimate} as
\begin{equation}\label{mathcalS}
\begin{aligned}
 &\sum_{\alpha,\beta=1}^d  \int_{\R^{dN}} F_N(t,V) \frac{1}{N^3}\sum_{\mathcal{S} }\Big[a\ast f(v^{i})-a(v^{i}-v^{j})
\Big]_{\alpha\beta}\Big[a\ast f(v^{i})-a(v^{i}-v^{k})
\Big]_{\alpha\beta} \ud V\\
\leq &\frac{C}{N} \int_{\R^{dN}} F_N(t,V)(1+|v^1|^4)\ud V\leq \frac{C}{N}(1+\mm_4(0)e^{\frac{8t}{N}})\lesssim\frac{1+\mm_4(0)}{N}e^{\frac{8T}{N}}.
\end{aligned}
\end{equation}
We expand the first term by the equality \eqref{component a} as
\begin{equation}\label{expansion}
\begin{aligned}
 & \sum_{\alpha,\beta=1}^d  \int_{\R^{dN}} F_N \frac{1}{N^3}\sum_{i\neq j\neq k}\Big[a\ast f(v^{i})-a(v^{i}-v^{j})
\Big]_{\alpha\beta}\Big[a\ast f(v^{i})-a(v^{i}-v^{k})
\Big]_{\alpha\beta} \ud V  \\     = &  \sum_{\alpha,\beta=1}^d  \int_{\R^{dN}} F_N\frac{1}{N^3}\sum_{i\neq j\neq k}\Big[\big(d+2v^i\cdot v^j-|v^j|^2\big)\delta_{\alpha\beta}-\big(\mathcal{E_\alpha}(t)\delta_{\alpha\beta}+v^i_\alpha v^j_\beta+v^j_\alpha v^i_\beta-v^j_\alpha v^j_\beta\big)\Big]\\
&\times \Big[\big(d+2v^i\cdot v^k-|v^k|^2\big)\delta_{\alpha\beta}-\big(\mathcal{E_\alpha}(t)\delta_{\alpha\beta}+v^i_\alpha v^k_\beta+v^k_\alpha v^i_\beta-v^k_\alpha v^k_\beta\big)\Big] \ud V\\
=:&I_{\alpha=\beta}+I_{\alpha\neq\beta},
 \end{aligned}
\end{equation}
where we separate the sum into the diagonal part $\alpha=\beta$ and the off-diagonal part $\alpha\neq\beta$ as follows:
$$
\begin{aligned}
I_{\alpha=\beta}:= &  \sum_{\alpha=1}^d   \int_{\R^{dN}} F_N\frac{1}{N^3}\sum_{i\neq j\neq k} \Big(d-|v^j|^2+2v^i\cdot v^j-2v^i_\alpha v^j_\alpha+(v^j_\alpha)^2-\mathcal{E_\alpha}(t)\Big)\\&\times\big(d-|v^k|^2+2v^i\cdot v^k-2v^i_\alpha v^k_\alpha+(v^k_\alpha )^2-\mathcal{E_\alpha}(t)\big)\ud V.
\end{aligned}
$$
We expand further to get
$$
\begin{aligned}
I_{\alpha=\beta}=&\sum_{\alpha=1}^d  \int_{\R^{dN}} \frac{F_N}{N^3}\sum_{i\neq j\neq k} \Big[\big(d-|v^j|^2\big)\big(d-|v^k|^2\big)+4\big(v^i\cdot v^j-v^i_\alpha v^j_\alpha\big)\big(v^i\cdot v^k-v^i_\alpha v^k_\alpha\big)\\&+\big((v^j_\alpha)^2-\mathcal{E_\alpha}(t)\big)\big((v^k_\alpha )^2-\mathcal{E_\alpha}(t)\big)+\big(d-|v^j|^2\big)\big(v^i\cdot v^k-v^i_\alpha v^k_\alpha\big)\\
&+2\big(d-|v^j|^2\big)\big((v^k_\alpha )^2-\mathcal{E_\alpha}(t)\big)+4\big(v^i\cdot v^j-v^i_\alpha v^j_\alpha\big)\big((v^k_\alpha )^2-\mathcal{E_\alpha}(t)\big)\Big]\ud V,
\end{aligned}
$$
and
$$
\begin{aligned}
I_{\alpha\neq\beta}&:=   \sum_{\alpha\neq\beta} \int_{\R^{dN}} F_N\frac{1}{N^3}\sum_{i\neq j\neq k} \Big(v^i_\alpha v^j_\beta+v^j_\alpha v^i_\beta-v^j_\alpha v^j_\beta\Big)\Big(v^i_\alpha v^k_\beta+v^k_\alpha v^i_\beta-v^k_\alpha v^k_\beta\Big)\ud V\\ 
&=   \sum_{\alpha\neq\beta} \int_{\R^{dN}} F_N\frac{1}{N^3}\sum_{i\neq j\neq k} \Big[2(v^i_\alpha )^2v^j_\beta v^k_\beta-2v^i_\alpha v^k_\alpha v^i_\beta v^j_\beta+v^j_\alpha v^j_\beta v^k_\alpha v^k_\beta\Big]\ud V,
\end{aligned}
$$
where the last equality is due to the symmetry  and the permutation. 

We aim to prove that
$$
I_{\alpha=\beta}+I_{\alpha\neq\beta}\leq \frac{C(T)}{N},
$$
which the dependency on time $T$ can be made explicitly later.
According to the Liouville equation in the weak form \eqref{weak form}, for any test function $\Phi\in C^2(\R^{dN})$ with polynomial growth, we have the following identity
$$
\begin{aligned}
&\frac{\ud }{\ud t} \int_{\R^{dN}}F_N(t,V)\Phi(V)\ud V\\
=&\sum_{i=1}^{N}\int_{\R^{dN}}\nabla^2_{v^{i}} \Phi(V): \frac{1}{N}\sum_{j=1}^Na(v^{i}-v^j) F_N(t,V)\ud V \\&+2\sum_{i=1}^{N} \int_{\R^{dN}} \nabla_{v^{i}} \Phi(V)\cdot \frac{1}{N}\sum_{j=1}^Nb(v^{i}-v^j)F_N(t,V)\ud V\\
=&\frac{1}{2N}\sum_{i,j=1}^{N}\int_{\R^{dN}}\big(\nabla^2_{v^{i}}\Phi(V)+\nabla^2_{v^{j}}\Phi(V)  \big): a(v^{i}-v^j) F_N(t,V)\ud V \\&+\frac{1}{N}\sum_{i,j=1}^{N} \int_{\R^{dN}} \big(\nabla_{v^{i}} \Phi(V)-\nabla_{v^{j}} \Phi(V)\big)\cdot b(v^{i}-v^j)F_N(t,V)\ud V,
\end{aligned}
$$
together with the convention $a(0)=0$ and $b(0)=0$ leads to that
\begin{equation}\label{weak form Liouville}
\begin{aligned}
&\frac{\ud }{\ud t} \int_{\R^{dN}}F_N(t,V)\Phi(V)\ud V\\
=&\frac{1}{2N}\sum_{i\neq j}\int_{\R^{dN}}\big(\nabla^2_{v^{i}}\Phi(V)+\nabla^2_{v^{j}}\Phi(V)  \big): a(v^{i}-v^j) F_N(t,V)\ud V \\&+\frac{1}{N}\sum_{i\neq j} \int_{\R^{dN}} \big(\nabla_{v^{i}} \Phi(V)-\nabla_{v^{j}} \Phi(V)\big)\cdot b(v^{i}-v^j)F_N(t,V)\ud V.
\end{aligned}    
\end{equation}
We observe that when $\vphi(v)=|v|^2$ or $\vphi(v)=v_\alpha$ for any $\alpha=1,\ldots,d$ for any $\alpha\neq\beta$, the identity holds for any $v^i\neq v^j\in \R^d$,
\begin{equation}\label{weak identity}
\frac{1}{2}\big(\nabla^2\vphi(v^i)+\nabla^2\vphi(v^j)  \big): a(v^{i}-v^j)+\big(\nabla \vphi(v^i)-\nabla \vphi(v^j)\big)\cdot b(v^{i}-v^j)=0.
\end{equation}
As a consequence, the following lemma holds.
\begin{lemma}\label{energy}
The average kinetic energy over particles is conservative, i.e.
$$\frac{\ud }{\ud t}\int_{\R^{dN}}F_N(t,V)\frac{1}{N}\sum_{i=1}^{N}|v^i|^2\ud V=0;$$
hence for any $t>0$, 
$$
\int_{\R^{dN}}F_N(t,V)\frac{1}{N}\sum_{i=1}^{N}|v^i|^2\ud V=d.
$$
\end{lemma}
Now we first let $\Phi(V)=\frac{1}{N^2}\sum_{k\neq l}\vphi_1(v^k)\vphi_2(v^l)$, we get
$$
\nabla_{v^i}\Phi(V)=\frac{1}{N^2}\sum_{l\neq i}\nabla \vphi_1(v^i)\vphi_2(v^l)+\frac{1}{N^2}\sum_{k\neq i}\vphi_1(v^k)\nabla \vphi_2(v^i);
$$
$$
\nabla_{v^i}^2\Phi(V)=\frac{1}{N^2}\sum_{l\neq i}\nabla ^2\vphi_1(v^i)\vphi_2(v^l)+\frac{1}{N^2}\sum_{k\neq i}\vphi_1(v^k)\nabla ^2\vphi_2(v^i).
$$
Substitute these into \eqref{weak form Liouville} to obtain that
$$
\begin{aligned}
&\frac{\ud }{\ud t} \int_{\R^{dN}}F_N(t,V)\Phi(V)\ud V\\
=&\frac{1}{2N^3}\sum_{i\neq j}\int_{\R^{dN}}\big(\nabla^2\vphi_1(v^i)\sum_{l\neq i}\vphi_2(v^l)+\nabla^2\vphi_2(v^i)\sum_{k\neq i}\vphi_1(v^k)  \big): a(v^{i}-v^j) F_N(t,V)\ud V \\&+\frac{1}{2N^3}\sum_{i\neq j}\int_{\R^{dN}}\big(\nabla^2\vphi_1(v^j)\sum_{l\neq i}\vphi_2(v^l)+\nabla^2\vphi_2(v^j)\sum_{k\neq j}\vphi_1(v^k)  \big): a(v^{i}-v^j) F_N(t,V)\ud V\\&+\frac{1}{N^3}\sum_{i\neq j} \int_{\R^{dN}} \big(\nabla \vphi_1(v^i)\sum_{l\neq i}\vphi_2(v^l)+\nabla \vphi_2(v^i)\sum_{k\neq i}\vphi_1(v^k)\big)\cdot b(v^{i}-v^j)F_N(t,V)\ud V\\&-\frac{1}{N^3}\sum_{i\neq j} \int_{\R^{dN}} \big(\nabla \vphi_1(v^j)\sum_{l\neq j}\vphi_2(v^l)+\nabla\vphi_2(v^j)\sum_{k\neq i}\vphi_1(v^k)\big)\cdot b(v^{i}-v^j)F_N(t,V)\ud V.
\end{aligned}    
$$
If both $\vphi_1$ and $\vphi_2$ satisfy the identity \eqref{weak identity}, then the majority of the terms above can be compensated by the symmetry property, and the remaining terms become
$$
\begin{aligned}
&\frac{\ud }{\ud t} \int_{\R^{dN}}F_N(t,V)\Phi(V)\ud V\\
=&\frac{1}{2N^3}\sum_{i\neq j}\int_{\R^{dN}}\big(\nabla ^2\vphi_1(v^i)\vphi_2(v^j)+\nabla ^2\vphi_2(v^i)\vphi_1(v^j)  \big): a(v^{i}-v^j) F_N(t,V)\ud V\\ 
+&\frac{1}{2N^3}\sum_{i\neq j}\int_{\R^{dN}}\big(\nabla ^2\vphi_1(v^j)\vphi_2(v^i)+\nabla ^2\vphi_2(v^j)\vphi_1(v^i)  \big): a(v^{i}-v^j) F_N(t,V)\ud V\\ &+\frac{1}{N^3}\sum_{i\neq j} \int_{\R^{dN}} \big(\nabla \vphi_1(v^i)\vphi_2(v^j)+\nabla \vphi_2(v^i)\vphi_1(v^j)\big)\cdot b(v^{i}-v^j)F_N(t,V)\ud V\\&-\frac{1}{N^3}\sum_{i\neq j} \int_{\R^{dN}} \big(\nabla \vphi_1(v^j)\vphi_2(v^i)+\nabla \vphi_2(v^j)\vphi_1(v^i)\big)\cdot b(v^{i}-v^j)F_N(t,V)\ud V.
\end{aligned}
$$
If we take $\vphi_1$ and $\vphi_2$ to be $|v|^2$ or $v_\alpha$ for any $\alpha=1,\ldots,d$, then we have
\begin{equation}\label{vphi12}
\begin{aligned}
\frac{\ud}{\ud t}\int_{\R^{dN}}F_N(t,V)\frac{1}{N^2}\sum_{k\neq l}\vphi_1(v^k)\vphi_2(v^l)\ud V\leq& \frac{C}{N} \int_{\R^{dN}} F_N(t,V)(1+|v^1|^4)\ud V\\\lesssim&\frac{1+\mm_4(0)}{N}e^{\frac{8t}{N}}.
\end{aligned}
\end{equation}
In the same spirit of the cancellation, if we take more test functions $\vphi_1$, $\vphi_2$, $\vphi_3$ and $\vphi_4$ to be $|v|^2$ or $v_\alpha$ for any $\alpha=1,\ldots,d$, then we have the estimate 
\begin{equation}\label{vphi123}
\begin{aligned}
\frac{\ud}{\ud t}\int_{\R^{dN}}F_N(t,V)\frac{1}{N^3}\sum_{k\neq l\neq m}\vphi_1(v^k)\vphi_2(v^l)\vphi_3(v^m)\ud V\sim \frac{C}{N}e^{\frac{C't}{N}};
\end{aligned}
\end{equation}
and
\begin{equation}\label{vphi1234}
\begin{aligned}
\frac{\ud}{\ud t}\int_{\R^{dN}}F_N(t,V)\frac{1}{N^4}\sum_{k\neq l\neq m\neq n}\vphi_1(v^k)\vphi_2(v^l)\vphi_3(v^m)\vphi_4(v^n)\ud V\sim \frac{C}{N}e^{\frac{C't}{N}},
\end{aligned}
\end{equation}
where constants depend on the moment estimate.
As directly consequences of \eqref{vphi12}, \eqref{vphi123} and \eqref{vphi1234}, since each integrand has no more than 4th-order in velocity,  the following statements hold:
\begin{claim}\label{alphacross}
For any $t>0$,  it holds
$$
\int_{\R^{dN}}F_N(t,V)\frac{1}{N^2}\sum_{i\neq j}v^i_\alpha v^j_\alpha\ud V\lesssim\frac{1+\mm_4(0)}{N}\int_0^t e^{\frac{8s}{N}}\ud s.
$$
\end{claim}
\begin{claim}\label{cross}
For any $t>0$,  it holds
$$
\int_{\R^{dN}}F_N(t,V)\frac{1}{N^2}\sum_{i\neq j}v^i\cdot v^j\ud V\lesssim\frac{1+\mm_4(0)}{N}\int_0^t e^{\frac{8s}{N}}\ud s.
$$
\end{claim}
\begin{claim}\label{energyenergy}
For any $t>0$,  it holds
$$
\int_{\R^{dN}}F_N(t,V)\frac{1}{N^2}\sum_{i\neq j}|v^i|^2|v^j|^2\ud V-d^2\lesssim\frac{1+\mm_4(0)}{N}\int_0^t e^{\frac{8s}{N}}\ud s.
$$
\end{claim}
\begin{claim}\label{alphacrossenergy}
For any $t>0$,  it holds
$$
\int_{\R^{dN}}F_N(t,V)\frac{1}{N^3}\sum_{i\neq j\neq k}v^i_\alpha v^j_\alpha| v^k|^2\ud V\lesssim\frac{1+\mm_4(0)}{N}\int_0^t e^{\frac{8s}{N}}\ud s.
$$
\end{claim}
\begin{claim}\label{crossenergy}
For any $t>0$,  it holds
$$
\int_{\R^{dN}}F_N(t,V)\frac{1}{N^2}\sum_{i\neq j\neq k}v^i\cdot v^j| v^k|^2\ud V\lesssim\frac{1+\mm_4(0)}{N}\int_0^t e^{\frac{8s}{N}}\ud s.
$$
\end{claim}

\begin{claim}\label{alphabetacrossalphacross}
For any $\alpha\neq\beta$,  it holds
$$
\int_{\R^{dN}}F_N(t,V)\frac{1}{N^4}\sum_{i\neq j\neq k\neq l}v^i_\alpha v^j_\beta v^k_\alpha v^l_\beta\ud V\lesssim\frac{1+\mm_4(0)}{N}\int_0^t e^{\frac{8s}{N}}\ud s.
$$ 
\end{claim}

\medskip

Next, we will deal with the terms with directional temperature. We start with the following lemma. 
\begin{lemma}\label{direction}
The average directional temperature over particles is comparable with the directional temperature for the limit equation, say
    $$
\int_{\R^{dN}}F_N(t,V)\frac{1}{N}\sum_{i=1}^{N}(v^i_\alpha)^2\ud V-\mathcal{E}_\alpha(t)\lesssim\frac{1+\mm_4(0)}{N}e^{-4dt}\int_0^t\int_0^s e^{\frac{8r}{N}}\ud r\ud s.
    $$
\end{lemma}
\begin{proof}[Proof of Lemma \ref{direction}]
  Notice that if we pick the test function $\vphi(V)=  \frac{1}{N}\sum_{i=1}^{N}(v^i_\alpha)^2$, then
  $$
\nabla_{v^i}\vphi(V)=(0,\ldots,\frac{2}{N}v^i_\alpha,\ldots,0), \quad \nabla^2_{v^i}\vphi(V)=\begin{cases}\frac{2}{N} \text{ at the } (\alpha,\alpha)\text{-th } \text{ entry},\\ 
 0 \text{  otherwise. }\end{cases}.
$$
Let 
$$
\Psi_\alpha(t)=: \int_{\R^{dN}}F_N(t,V)\frac{1}{N}\sum_{i=1}^{N}(v^i_\alpha)^2\ud V,
$$
which satisfies that $\sum_{\alpha=1}^{d}\mathcal{E}_\alpha(t)=\sum_{\alpha=1}^{d}\Psi_\alpha(t)=d$ and that $\Phi_\alpha(0)=\me_\alpha(0)$. Take the time derivative and apply \eqref{weak form Liouville} to obtain that 
$$
\begin{aligned}
\frac{\ud}{\ud t} \Psi_\alpha(t):=&\frac{2}{N^2}\sum_{i\neq j}\int_{\R^{dN}}\Big(|v^i-v^j|^{2}-(v^i-v^j)_\alpha^2\Big) F_N(t,V)\ud V\\& -(d-1)\frac{2}{N^2}\sum_{i\neq j}\int_{\R^{dN}} (v^i-v^j)_\alpha^2 F_N(t,V)\ud V\\ 
=&\frac{2}{N^2}\sum_{i\neq j}\int_{\R^{dN}}(|v^i|^2+|v^j|^2-2v^i\cdot v^j) F_N(t,V)\ud V\\&-\frac{2d}{N^2}\sum_{i\neq j} \int_{\R^{dN}} \big((v^i_\alpha)^2+(v^j_\alpha)^2-2v^i_\alpha v^j_\alpha \big)F_N(t,V)\ud V\\
=&4d-4d\Psi_\alpha(t)
\\&- \frac{4}{N^2}\sum_{i\neq j}\int_{\R^{dN}}F_N(t,V)v^i\cdot v^j\ud V+ \frac{4d}{N^2}\sum_{i\neq j}\int_{\R^{dN}}F_N(t,V)v^i_\alpha  v^j_\alpha\ud V.
\end{aligned}
$$
Notice that the following quantity with the minus sign has the same control as Claim \ref{alphacross} that 
$$
\frac{\ud}{\ud t}\Big(- \frac{4}{N^2}\sum_{i\neq j}\int_{\R^{dN}}F_N(t,V)v^i\cdot v^j\ud V\Big)\lesssim \frac{1+\mm_4(0)}{N}e^{\frac{8t}{N}}.
$$
Also recall that the directional temperature satisfies the following ODE: 
$$
\begin{aligned}
  \frac{\ud}{\ud t} \mathcal{E}_\alpha(t)=4 d-4d \mathcal{E}_\alpha(t),
\end{aligned}
$$
hence combining with Claim \ref{cross} we deduce that 
$$
\begin{aligned}
\frac{\ud}{\ud t}(\Psi_\alpha(t)-\mathcal{E}_\alpha(t))\leq&-4d(\Psi_\alpha(t)-\mathcal{E}_\alpha(t))+\frac{C(1+\mm_4(0))}{N}\int_0^t e^{\frac{8s}{N}}\ud s.
\end{aligned}
$$
Applying the Gr\"onwall's inequality and the initial identity $\Psi_\alpha(0)=\me_\alpha(0)$ implies that
$$
\begin{aligned}
&\Psi_\alpha(t)-\mathcal{E}_\alpha(t)\lesssim \frac{1+\mm_4(0)}{N}e^{-4dt}\int_0^t\int_0^s e^{\frac{8r}{N}}\ud r\ud s.
\end{aligned}
$$
\end{proof}
Now we take the test function $\Phi(V)=\frac{1}{N^2}\sum_{k\neq l}(v^k_\alpha)^2\vphi (v^l)$ for some direction $\alpha=1,...,d$ with some one-variable test function $\vphi\in C^2(\R^d)$, then we compute the gradient and Hessian of $\Phi$ respectively as
$$
\nabla_{v^i}\Phi(V)=\frac{2}{N^2}(0,\ldots,v^i_\alpha,\ldots,0)\sum_{l\neq i}\vphi(v^l)+\frac{1}{N^2}\sum_{k\neq i}(v^k_\alpha)^2\nabla \vphi(v^i);
$$
$$
\nabla_{v^i}^2\Phi(V)=\frac{1}{N^2}\sum_{k\neq i}(v^k_\alpha)^2\nabla^2\vphi(v^i)+\frac{2}{N^2}\sum_{l\neq i}\vphi(v^l)\times\begin{cases}1 \text{ at the } (\alpha,\alpha)\text{-th } \text{ entry}\\ 
 0 \text{  otherwise }\end{cases},
$$
where the last matrix is non-zero only at the $(\alpha,\alpha)$-th entry.
Substitute these into \eqref{weak form Liouville} to obtain that
$$
\begin{aligned}
\frac{\ud \Psi_\alpha(t)}{\ud t}=&\frac{\ud }{\ud t} \int_{\R^{dN}}F_N(t,V)\frac{1}{N^2}\sum_{k\neq l}(v^k_\alpha)^2\vphi (v^l)\ud V\\
=&\frac{1}{2N^3}\sum_{i\neq j}\int_{\R^{dN}}\big(\nabla^2\vphi(v^i)\sum_{k\neq i}(v^k_\alpha)^2+\nabla^2\vphi(v^j)\sum_{k\neq j}(v^k_\alpha)^2  \big): a(v^{i}-v^j) F_N(t,V)\ud V \\&+\frac{1}{N^3}\sum_{i\neq j}\int_{\R^{dN}}\big(\sum_{l\neq i}\vphi(v^l)+\sum_{l\neq j}\vphi(v^l)  \big) \big(|v^i-v^j|^2-(v^i-v^j)_\alpha^2\big)F_N(t,V)\ud V\\&+\frac{1}{N^3}\sum_{i\neq j} \int_{\R^{dN}} \big(\nabla \vphi(v^i)\sum_{k\neq i}(v^k_\alpha)^2-\nabla \vphi(v^j)\sum_{k\neq j}(v^k_\alpha)^2\big)\cdot b(v^{i}-v^j)F_N(t,V)\ud V\\&-\frac{2(d-1)}{N^3}\sum_{i\neq j} \int_{\R^{dN}} \big(v_\alpha^i\sum_{l\neq j}\vphi(v^l)-v_\alpha^j\sum_{k\neq i}\vphi(v^k)\big)(v^i-v^j)_\alpha F_N(t,V)\ud V.
\end{aligned}    
$$
If the test function $\vphi$ satisfies the identity \eqref{weak identity}, then the majority of the terms above can be compensated by symmetry, and the remaining terms become
$$
\begin{aligned}
&\frac{\ud }{\ud t} \int_{\R^{dN}}F_N(t,V)\frac{1}{N^2}\sum_{k\neq l}(v^k_\alpha)^2\vphi (v^l)\ud V\\
=&\frac{1}{2N^3}\sum_{i\neq j}\int_{\R^{dN}}\big(\nabla^2\vphi(v^i)(v^j_\alpha)^2+\nabla^2\vphi(v^j)(v^i_\alpha)^2  \big): a(v^{i}-v^j) F_N(t,V)\ud V\\ 
&+\frac{1}{N^3}\sum_{i\neq j} \int_{\R^{dN}} \big(\nabla\vphi(v^i)(v^j_\alpha)^2-\nabla \vphi(v^j)(v^i_\alpha)^2\big)\cdot b(v^{i}-v^j)F_N(t,V)\ud V\\
&+\frac{1}{N^3}\sum_{i\neq j}\int_{\R^{dN}}\big(\vphi(v^j)+\vphi(v^i)  \big) \big(|v^i-v^j|^2-(v^i-v^j)_\alpha^2\big)F_N(t,V)\ud V 
\\&-\frac{2(d-1)}{N^3}\sum_{i\neq j} \int_{\R^{dN}} \big(v_\alpha^i\vphi(v^i)-v_\alpha^j\vphi(v^j)\big)(v^i-v^j)_\alpha F_N(t,V)\ud V\\
&+\frac{2}{N^3}\sum_{l\neq i\neq j}\int_{\R^{dN}} \big(|v^i-v^j|^2-d(v^i-v^j)_\alpha^2\big)\vphi(v^l)F_N(t,V)\ud V. 
\end{aligned}
$$
If we take $\vphi$ to be $|v|^2$ or $v_\alpha$ for any $\alpha=1,\ldots,d$, then first four terms at the right-hand side can be bounded by the $4$-th moment of $F_N$ by Proposition \ref{moment estimate} with the order $1/N$; while the last term is controlled by
$$
\begin{aligned}
&\frac{2}{N^3}\sum_{l\neq i\neq j}\int_{\R^{dN}} \big(|v^i-v^j|^2-d(v^i-v^j)_\alpha^2\big)\vphi(v^l)F_N(t,V)\ud V \\
=& \frac{2}{N^3}\sum_{l\neq i\neq j}\int_{\R^{dN}} \big(|v^i|^2+|v^j|^2-2v^i\cdot v^j+2dv^i_\alpha v^j_\alpha\big)\vphi(v^l)F_N(t,V)\ud V\\
&- \frac{2d}{N^3}\sum_{i\neq j}\int_{\R^{dN}} \big((v^i_\alpha)^2+(v^j_\alpha)^2\big)\sum_{l\neq i\neq j}\vphi(v^l)F_N(t,V)\ud V\\
\leq&\frac{C(1+\mm_4(0))}{N}\int_0^t e^{\frac{8s}{N}}\ud s+4d\int_{\R^d}\vphi(v)f_0(v)\ud v\\&-4d\int_{\R^{dN}}F_N(t,V)\frac{1}{N^2}\sum_{k\neq l}(v^k_\alpha)^2\vphi (v^l)\ud V.
\end{aligned}
$$
And recall that the directional temperature satisfies the following ODE: 
$$
\begin{aligned}
  \frac{\ud}{\ud t} \mathcal{E}_\alpha(t)=4 d-4d \mathcal{E}_\alpha(t),
\end{aligned}
$$
which still holds after multiplying with some constant:
$$
\begin{aligned}
  \frac{\ud}{\ud t} \Big(\me_\alpha(t)\int_{\R^d}\vphi(v)f_0(v)\ud v\Big)=4 d\int_{\R^d}\vphi(v)f_0(v)\ud v-4d \Big(\me_\alpha(t)\int_{\R^d}\vphi(v)f_0(v)\ud v\Big),
\end{aligned}
$$
we deduce that 
$$
\begin{aligned}
  &\frac{\ud}{\ud t} \Big(\Psi_\alpha(t)-\me_\alpha(t)\int_{\R^d}\vphi(v)f_0(v)\ud v\Big)\\\leq&\frac{C(1+\mm_4(0))}{N}\int_0^t e^{\frac{8s}{N}}\ud s-4d \Big(\Psi_\alpha(t)-\me_\alpha(t)\int_{\R^d}\vphi(v)f_0(v)\ud v\Big).
\end{aligned}
$$
Applying the Gr\"onwall's inequality and the fact
$$
\Psi_\alpha(0)=\me_\alpha(0)\int_{\R^d}\vphi(v)f_0(v)\ud v
$$
implies that

\begin{equation}\label{Psialpha}
\begin{aligned}
\Psi_\alpha(t)-\me_\alpha(t)\int_{\R^d}\vphi(v)f_0(v)\ud v\lesssim \frac{1+\mm_4(0)}{N}e^{-4dt}\int_0^t\int_0^s e^{\frac{8r}{N}}\ud r\ud s.
\end{aligned}
\end{equation}

In the same spirit as the proof of \eqref{Psialpha}, we can prove that following several statements hold, which all involve the directional temperature:
\begin{claim}\label{energydirection}
For any $t>0$,  it holds
$$\int_{\R^{dN}}F_N(t,V)\frac{1}{N^2}\sum_{i\neq j}|v^i|^2(v_\alpha^j)^2 \ud V- \me_\alpha(t) d\lesssim \frac{1+\mm_4(0)}{N}e^{-4dt}\int_0^t\int_0^s e^{\frac{8r}{N}}\ud r\ud s.$$
\end{claim}

\begin{claim}\label{alphacrossdirection}
For any $t>0$,  it holds
$$\int_{\R^{dN}}F_N(t,V)\frac{1}{N^3}\sum_{i\neq j\neq k}v^i_\alpha v^j_\alpha( v^k_\alpha)^2\ud V\lesssim \frac{1+\mm_4(0)}{N}e^{-4dt}\int_0^t\int_0^s e^{\frac{8r}{N}}\ud r\ud s.$$
\end{claim}
\begin{claim}\label{crossdirection}
For any $t>0$,  it holds
$$\int_{\R^{dN}}F_N(t,V)\frac{1}{N^3}\sum_{i\neq j\neq k}v^i\cdot v^j( v^k_\alpha)^2\ud V\lesssim \frac{1+\mm_4(0)}{N}e^{-4dt}\int_0^t\int_0^s e^{\frac{8r}{N}}\ud r\ud s.$$
\end{claim}

\begin{claim}\label{iijk}
For any $\alpha\neq\beta$ and $t>0$,  it holds

$$\int_{\R^{dN}}F_N(t,V)\frac{1}{N^3}\sum_{i\neq j
\neq k}v^i_\beta v^j_\beta(v^k_\alpha)^2\ud V\lesssim \frac{1+\mm_4(0)}{N}e^{-4dt}\int_0^t\int_0^s e^{\frac{8r}{N}}\ud r\ud s.$$    
\end{claim}

Slightly different from claims above, we have the following lemma involving the double directional temperature.
\begin{lemma}\label{directiondirection}
It holds for any $t>0$ that
$$
\begin{aligned}
&\int_{\R^{dN}}F_N(t,V)\frac{1}{N^2}\sum_{i\neq j}(v^i_\alpha)^2(v^j_\alpha)^2 \ud V\\\leq&  \me_\alpha(t)^2+\frac{C(1+\mm_4(0))}{N}e^{-8dt}\int_0^te^{-4d\tau}\int_0^\tau\int_0^s e^{\frac{8r}{N}}\ud r\ud s\ud \tau.
\end{aligned}
$$
\end{lemma}

\begin{proof}[Proof of  Lemma \ref{directiondirection}]
We pick the test function $\Phi(V)=\frac{1}{N^2}\sum_{k\neq l}(v^k_\alpha)^2(v^l_\alpha)^2$ in \eqref{weak form Liouville} to obtain that
$$
\begin{aligned}
\frac{\ud}{\ud t}\Psi_\alpha(t):=&\frac{\ud}{\ud t} \int_{\R^{dN}}F_N(t,V)\frac{1}{N^2}\sum_{k\neq l}(v^k_\alpha)^2(v^l_\alpha)^2\ud V\\
=&\frac{4}{N^3}\sum_{i\neq j\neq k}\int_{\R^{dN}}(v^k_\alpha )^2|v^i-v^j|^2F_N\ud V-\frac{4d}{N^3}\sum_{i,j,k}\int_{\R^{dN}}(v^k_\alpha )^2(v^i_\alpha-v^j_\alpha)^2F_N\ud V\\
=&\frac{4}{N^2}\sum_{i\neq j\neq k}\int_{\R^{dN}}(v^k_\alpha )^2(|v^i|^2-2v^i\cdot v^j+|v^j|^2)\ud V\\+&\frac{8d}{N^3}\sum_{i\neq j\neq k}\int_{\R^{dN}}(v^k_\alpha )^2 v^i_\alpha v^j_\alpha F_N \ud V-\frac{4d}{N^2}\sum_{i\neq j\neq k}\int_{\R^{dN}}(v^k_\alpha )^2\big((v^i_\alpha)^2+(v^j_\alpha)^2 \big)F_N \ud V,
\end{aligned}
$$  
where thanks to Claim \ref{energydirection}, Claim \ref{alphacrossdirection} and Claim \ref{crossdirection}, the first two terms together can be bounded by
$$
\begin{aligned}
&\frac{4}{N^2}\sum_{i\neq j\neq k}\int_{\R^{dN}}(v^k_\alpha )^2(|v^i|^2-2v^i\cdot v^j+|v^j|^2)\ud V+\frac{8d}{N^3}\sum_{i\neq j\neq k}\int_{\R^{dN}}(v^k_\alpha )^2 v^i_\alpha v^j_\alpha F_N \ud V  \\
\leq & 8d\me_\alpha(t) +\frac{C(1+\mm_4(0))}{N}e^{-4dt}\int_0^t\int_0^s e^{\frac{8r}{N}}\ud r\ud s;
\end{aligned}
$$
while the second term is nothing but $\Psi_\alpha(t)$ itself, namely 
$$
\begin{aligned}
\frac{\ud}{\ud t}\Psi_\alpha(t)\leq  8d\me_\alpha (t)-8d\Psi_\alpha(t)+\frac{C(1+\mm_4(0))}{N}e^{-4dt}\int_0^t\int_0^s e^{\frac{8r}{N}}\ud r\ud s.
\end{aligned}
$$
Notice that
$$
\begin{aligned}
   \frac{\ud}{\ud t}\me_\alpha(t)^2=2\me_\alpha(t)\frac{\ud}{\ud t} \me_\alpha(t)=8d \me_\alpha(t)-8d \mathcal{E}_\alpha(t)^2,
\end{aligned}
$$
then we have
$$
 \frac{\ud}{\ud t}\big(\Psi_\alpha(t)-\mathcal{E}_\alpha(t)^2\big)\leq -8d\big(\Psi_\alpha(t)-\mathcal{E}_\alpha(t)^2\big)+\frac{C(1+\mm_4(0))}{N}e^{-4dt}\int_0^t\int_0^s e^{\frac{8r}{N}}\ud r\ud s
$$
with $\Psi_\alpha(0)=\mathcal{E}_\alpha(0)^2$. 
We conclude the desired result by the Gr\"onwall's inequality.
\end{proof}

\medskip

Finally, we need to deal with the terms containing $v^i_\alpha v^i_\beta$ with $\alpha\neq\beta$.
\begin{lemma}\label{ijik}
For any $\alpha\neq\beta$ and $t>0$,  it holds

    $$
\int_{\R^{dN}}\frac{1}{N^3}\sum_{i\neq j\neq k}v^i_\alpha v^i_\beta v^j_\alpha v^k_\beta\ud V\lesssim \frac{1+\mm_4(0)}{N}e^{-4dt}\int_0^t\big(e^{\frac{8s}{N}}+\int_0^se^{\frac{8r}{N}}\ud r\big)\ud s .
$$ 
\end{lemma}
\begin{proof}[Proof of Lemma \ref{ijik}]
We pick the test function $\vphi(V)=\frac{1}{N^3}\sum_{k\neq l\neq m}v^k_\alpha v^k_\beta v^l_\alpha v^m_\beta$ to obtain that
$$
\begin{aligned}
\nabla_{v^i}\vphi(V)=&\frac{1}{N^3}(0,\dots,v^i_\beta,\ldots,v^i_\alpha,\ldots0)\sum_{l\neq m\neq i}v^l_\alpha v^m_\beta\\&+\frac{1}{N^3}\sum_{k\neq l\neq i}v^k_\alpha v^k_\beta (0,\dots,v^l_\alpha,\ldots,0)++\frac{1}{N^3}\sum_{k\neq m\neq i}v^k_\alpha v^k_\beta (0,\dots,v^m_\beta,\ldots,0)
,
\end{aligned}
$$
and
$$
\nabla_{v^i}^2\vphi(V)=\frac{1}{N^3}\sum_{l\neq m\neq i} v^l_\alpha v^m_\beta\times\begin{cases}1 \text{ at the } (\alpha,\beta)\text{-th }\mbox{ and } (\beta,\alpha)\text{-th } \text{ entry}\\ 
 0 \text{  otherwise}\end{cases}.
$$
Hence we have
$$
\begin{aligned}
&\frac{\ud}{\ud t}\int_{\R^{dN}}F_N(t,V)\frac{1}{N^3}\sum_{k\neq l\neq m}v^k_\alpha v^k_\beta v^l_\alpha v^m_\beta\ud V\\
\leq 
&-\frac{2d}{N^3}\sum_{i\neq j\neq k\neq l}\int_{\R^{dN}}v^k_\alpha v^l_\beta(v^i_\beta-v^j_\beta)(v^i_\alpha-v^j_\alpha)F_N\ud V+\frac{C(1+\mm_4(0))}{N}e^{\frac{8t}{N}}\\
\leq &-\frac{4d}{N^3}\sum_{i\neq k\neq l}\int_{\R^{dN}}v^k_\alpha v^l_\beta v^i_\beta v^i_\alpha F_N\ud V+\frac{4d}{N^4}\sum_{i\neq j\neq k\neq l}\int_{\R^{dN}}v^k_\alpha v^l_\beta v^j_\beta v^i_\alpha F_N\ud V\\&+\frac{C(1+\mm_4(0))}{N}e^{\frac{8t}{N}}.
\end{aligned}
$$  

The second term on the right-hand side can be controlled thanks to Claim \ref{alphabetacrossalphacross}. And recall our assumption that $\me_{\alpha\beta}(0)=0$, thus the desired lemma is deduced from the Gr\"onwall's inequality.
\end{proof}

The last technical lemma we need to prove is as follows, which involves double $v^i_\alpha v^i_\beta$ structure.
\begin{lemma}\label{alphabetacrossalphabetacross}
For any $\alpha\neq\beta$ and $t>0$,  it holds that

$$
\begin{aligned}
&\int_{\R^{dN}}F_N(t,V)\frac{1}{N^2}\sum_{i\neq j}v^i_\alpha v^i_\beta v^j_\alpha v^j_\beta\ud V\\\lesssim &\frac{1+\mm_4(0)}{N}\int_0^t\Big[e^{\frac{8\tau}{N} }+e^{-4d\tau}\int_0^\tau\big(e^{\frac{8s }{N}}+\int_0^s e^{\frac{8r }{N}}\ud r\big)\ud s\Big]\ud \tau.
\end{aligned}
$$ 
 \end{lemma}

\begin{proof}[Proof of Lemma \ref{alphabetacrossalphabetacross}]
Let $\Phi(V)=\frac{1}{N^2}\sum_{k\neq l}v^k_\alpha v^k_\beta v^l_\alpha v^l_\beta$ with components $\alpha< \beta$, we obtain
$$
\nabla_{v^i}\vphi(V)=\frac{2}{N^2}(0,\dots,v^i_\beta,\ldots,v^i_\alpha,\ldots0)\sum_{k\neq i}v^k_\alpha v^k_\beta ,
$$
and
$$
\nabla_{v^i}^2\vphi(V)=\frac{2}{N^2}\sum_{k\neq i}v^k_\alpha v^k_\beta\times\begin{cases}1 \text{ at the } (\alpha,\beta)\text{-th }\mbox{ and } (\beta,\alpha)\text{-th } \text{ entries}\\ 
 0 \text{  otherwise}\end{cases}.
$$
Hence we have $$
\begin{aligned}
&\frac{\ud}{\ud t}\int_{\R^{dN}}F_N(t,V)\frac{1}{N^2}\sum_{k\neq l}v^k_\alpha v^k_\beta v^l_\alpha v^l_\beta\ud V
\\\leq&-\frac{2d}{N^3}\sum_{i\neq j\neq k}\int_{\R^{dN}}v^k_\alpha v^k_\beta(v^i_\alpha-v^j_\alpha)(v^i_\beta-v^j_\beta)F_N\ud V+\frac{C(1+\M_4(0))}{N}e^{\frac{8t}{N}}\\
\leq &\frac{4d}{N^3}\sum_{i\neq k}\int_{\R^{dN}}v^k_\alpha v^k_\beta v^i_\alpha v^j_\beta F_N\ud V-\frac{4d}{N^2}\sum_{i\neq k}\int_{\R^{dN}}v^k_\alpha v^k_\beta v^i_\alpha v^i_\beta F_N\ud V+\frac{C(1+\M_4(0))}{N}e^{\frac{8t}{N}}.\end{aligned}
$$  
Thanks to Lemma \ref{ijik} and the Gr\"onwall's inequality, for any $\alpha\neq\beta$, it holds that
$$
\begin{aligned}
&\int_{\R^{dN}}F_N(t,V)\frac{1}{N^2}\sum_{i\neq j}v^i_\alpha v^i_\beta v^j_\alpha v^j_\beta\ud V\\\leq &\frac{C(1+\mm_4(0))}{N}\int_0^t\Big[e^{\frac{8\tau}{N} }+e^{-4d\tau}\int_0^\tau\big(e^{\frac{8s }{N}}+\int_0^s e^{\frac{8r }{N}}\ud r\big)\ud s\Big]\ud \tau.
\end{aligned}
$$ 

\end{proof}

\medskip

So far, we have finished all the ingredient for our proof of the Law of Large Numbers theorem. Now recall the expansion \eqref{expansion}, the off-diagonal part $I_{\alpha\neq\beta}$ can be bounded by Claim \ref{iijk}, Lemma \ref{ijik} and Lemma \ref{alphabetacrossalphabetacross}
by
\begin{equation}\label{off-dig part}
\begin{aligned}
I_{\alpha\neq\beta}(t)\lesssim\frac{1+\mm_4(0)}{N}\Big[&\int_0^t\Big(e^{\frac{8\tau}{N} }+e^{-4d\tau}\int_0^\tau\big(e^{\frac{8s }{N}}+\int_0^s e^{\frac{8r }{N}}\ud r\big)\ud s\Big)\ud \tau\\&+e^{-4dt}\int_0^t\big(e^{\frac{8s }{N}}+\int_0^s e^{\frac{8r }{N}}\ud r\big)\ud s\Big];
\end{aligned}
\end{equation}
while the diagonal part satisfies 
\begin{equation}\label{dig part}
\begin{aligned}
I_{\alpha=\beta}(t)\lesssim\frac{1+\mm_4(0)}{N}\Big[&\int_0^te^{\frac{8\tau}{N} }\ud \tau+e^{-4dt}\int_0^t\big(e^{\frac{8s }{N}}+\int_0^s e^{\frac{8r }{N}}\ud r\big)\ud s\\&+e^{-8dt}\int_0^te^{-4d\tau}\int_0^\tau\int_0^s e^{\frac{8r}{N}}\ud r\ud s\ud \tau\Big],
\end{aligned}
\end{equation}
which is directly from Corollary \ref{ee}-\ref{aa} below.

\begin{cor}\label{ee}
Lemma \ref{energy} and Claim \ref{energyenergy} imply that
$$\int_{\R^{dN}}F_N(t,V)\frac{1}{N^2}\sum_{j\neq k}\big(d-|v^j|^2\big)\big(d-|v^k|^2\big)\ud V\lesssim\frac{1+\mm_4(0)}{N}\int_0^t e^{\frac{8s}{N}}\ud s.$$
\end{cor}

\begin{cor}\label{ec}
Claim \ref{alphacross}, Claim \ref{cross}, Claim \ref{alphacrossenergy} and Claim \ref{crossenergy} imply that
$$\int_{\R^{dN}}F_N(t,V)\frac{1}{N^3}\sum_{i\neq j\neq k}\big(d-| v^j|^2\big)\big(2v^i\cdot v^k-2v^i_\alpha v^k_\alpha\big)\ud V\lesssim\frac{1+\mm_4(0)}{N}\int_0^t e^{\frac{8s}{N}}\ud s.$$
\end{cor}

\begin{cor}\label{ea}
Lemma \ref{energy}, Lemma \ref{direction} and Claim \ref{energydirection} imply that
$$
\begin{aligned}
&\int_{\R^{dN}}F_N(t,V)\frac{1}{N^3}\sum_{i\neq j\neq k}\big(d-| v^j|^2\big)\big((v^k_\alpha)^2-\mathcal{E}_\alpha(t)\big)\ud V\\\lesssim&\frac{1+\mm_4(0)}{N}\Big(e^{-4dt}\int_0^t\int_0^s e^{\frac{8r}{N}}\ud r\ud s+\int_0^t e^{\frac{8s}{N}}\ud s\Big).
\end{aligned}
$$
\end{cor}

\begin{cor}\label{ca}
Claim \ref{alphacross} Claim \ref{cross}, Claim \ref{alphacrossdirection} and Claim \ref{crossdirection} imply that
$$
\begin{aligned}
&\int_{\R^{dN}}F_N(t,V)\frac{1}{N^3}\sum_{i\neq j\neq k}\big(2v^i\cdot v^j- 2v^i_\alpha v^j_\alpha\big)\big((v^k_\alpha)^2-\mathcal{E}_\alpha(t)\big)\ud V\\\lesssim&\frac{1+\mm_4(0)}{N}\Big(e^{-4dt}\int_0^t\int_0^s e^{\frac{8r}{N}}\ud r\ud s+\int_0^t e^{\frac{8s}{N}}\ud s\Big).
\end{aligned}
$$

\end{cor}

\begin{cor}\label{cc}
Claim \ref{alphacrossdirection} and Claim \ref{ijik} imply that
$$\begin{aligned}&\int_{\R^{dN}}F_N(t,V)\frac{1}{N^3}\sum_{i\neq j\neq k}\big(v^i\cdot v^j- v^i_\alpha v^j_\alpha\big)\big(v^i\cdot v^k- v^i_\alpha v^k_\alpha\big)\ud V\\\lesssim &\frac{1+\mm_4(0)}{N}e^{-4dt}\int_0^t\big(e^{\frac{8s}{N}}+\int_0^se^{\frac{8r}{N}}\ud r\big)\ud s.
\end{aligned}
$$
\end{cor}

\begin{cor}\label{aa}
Lemma \ref{direction} and Lemma \ref{directiondirection} imply that
 $$
\begin{aligned}
&\int_{\R^{dN}}F_N(t,V)\frac{1}{N^2}\sum_{ j\neq k}\big((v^j_\alpha)^2-\mathcal{E}_\alpha(t)\big)\big((v^k_\alpha)^2-\mathcal{E}_\alpha(t)\big)\ud V\\\lesssim&\frac{1+\mm_4(0)}{N}\Big(e^{-4dt}\int_0^t\int_0^s e^{\frac{8r}{N}}\ud r\ud s+e^{-8dt}\int_0^te^{-4d\tau}\int_0^\tau\int_0^s e^{\frac{8r}{N}}\ud r\ud s\ud \tau\Big).
\end{aligned}
$$
\end{cor}
Hence the diagonal part \eqref{dig part} and the off-diagonal part \eqref{off-dig part} together yield that, 
$$
\begin{aligned}
&\sup_{t\in[0,T]}\big(I_{\alpha\neq\beta}(t)+I_{\alpha=\beta}(t)\big)\lesssim\frac{1+\mm_4(0)}{N}\Big[\int_0^t\Big(e^{\frac{8\tau}{N} }+e^{-4d\tau}\int_0^\tau\big(e^{\frac{8s }{N}}+\int_0^s e^{\frac{8r }{N}}\ud r\big)\ud s\Big)\ud \tau\\&+e^{-4dt}\int_0^t\big(e^{\frac{8s }{N}}+\int_0^s e^{\frac{8r }{N}}\ud r\big)\ud s+e^{-8dt}\int_0^te^{-4d\tau}\int_0^\tau\int_0^s e^{\frac{8r}{N}}\ud r\ud s\ud \tau\Big]\\&\leq \frac{C(1+\mm_4(0))(1+T)}{N},
\end{aligned}
$$
where the  integral of time has the order 
$$
\int_0^Te^{\frac{8s}{N}}\ud s=\frac{N}{8}(e^{\frac{8T}{N}}-1)\leq CT, \quad\text{when}\quad T\sim O(N).
$$
We conclude the proof of the Law of Large Numbers Theorem \ref{LLN} by combining with the estimate \eqref{mathcalS}.

\subsection*{Acknowledgements}
This work was partially supported by the National Key R\&D
Program of China, Project Number 2021YFA1002800. 
JAC was supported by the Advanced Grant Nonlocal-CPD (Nonlocal PDEs for Complex Particle dynamics: Phase Transitions, Patterns and Synchronization) of the European Research Council Executive Agency (ERC) under the European Union’s Horizon 2020 research and innovation programme (grant agreement No. 883363). JAC was also partially supported by the “Maria de Maeztu” Excellence Unit IMAG, reference CEX2020-001105-M, funded by MCIN/AEI/10.13039/501100011033/ and the EPSRC grant numbers EP/T022132/1 and EP/V051121/1. PEJ was partially
supported by NSF DMS Grants 2205694 and 2219297. ZW was partially supported by NSFC grant No.12171009 and  Young
Elite Scientist Sponsorship Program by China Association for Science and Technology
(CAST) No. YESS20200028.

\appendix
\section{ Moment Estimate of the Liouville Equation}\label{appendix}
In the appendix, we estimate the $p$-th order moment ($p>2$) of the first marginal of solution of the Liouville equation \eqref{master}. We denote the average $p$-th order moment as $\mm_p(t)$, and by the exchangeability,
$$
\mm_p(t):=\int_{\R^{d}}F_{N,1}(t,v^1)|v^1|^p\ud v^1 =\int_{\R^{dN}}F_N(t,V)\frac{1}{N}\sum_{k=1}^N|v^k|^p\ud V, $$ with $\mm_p(0)=\int_{\R^{d}}f_0(v)|v|^p\ud v.$
We have the following estimate.
\begin{prop}\label{moment estimate}
If we assume $\M_p(0)$ is finite, i.e., $f_0$ has finite $p$-th order moment, then the following estimate holds
 $$
\mm_p(t)\leq \mm_p(0)\left(\frac{p+d-3}{d-1}\right)^{\frac{p}{2}}e^{\frac{p(p-2)t}{N}}.
$$   
\end{prop}
\begin{proof}
By the exchangeability and the weak form of the collision operator \eqref{weak form}, we have
\begin{equation}\label{otherhand}
\begin{aligned}
&\frac{\ud}{\ud t} \int_{\R^{dN}}F_N(t,V)\frac{1}{N}\sum_{k=1}^N|v^k|^p\ud V\\
=&\frac{p(d-1)}{2N^2}\sum_{i, j}\int_{\R^{dN}}\big(|v^{i}|^2|v^{j}|^{p-2}+|v^{j}|^2|v^{i}|^{p-2}\big)F_N\ud V\\&-\frac{p(d-1)}{2N^2}\sum_{i,j}\int_{\R^{dN}}\big(|v^{i}|^p+|v^{j}|^p\big)F_N\ud V\\
&+\frac{p(p-2)}{2N^2}\sum_{i, j}\int_{\R^{dN}}\big(|v^{i}|^2|v^{j}|^{p-2}+|v^{j}|^2|v^{i}|^{p-2}\big)F_N\ud V\\&-\frac{p(p-2)}{2N^2}\sum_{i, j}\int_{\R^{dN}}(v^i\cdot v^j)^2(|v^{j}|^{p-4}+|v^{i}|^{p-4})F_N\ud V\\
\leq& \frac{p(p+d-3)}{N}\sum_{ j}\int_{\R^{dN}}\big(\frac{1}{N}\sum_i|v^{i}|^2\big)|v^{j}|^{p-2}F_N(t,V)\ud V\\&-\frac{p(d-1)}{N}\sum_{i}\int_{\R^{dN}}|v^{i}|^pf_N(t,V)\ud V.
\end{aligned}    
\end{equation}
Let $\vphi(V)=\frac{1}{N}\sum_{k=1}^{N}|v^k|^2$, then by the H\"{o}lder's inequality, 
$$
\begin{aligned}
&\frac{1}{N}\sum_{j=1}^N\int_{\R^{dN}}\vphi(V)|v^{j}|^{p-2}F_N(t,V)\ud V\\
\leq &\frac{1}{N}\sum_{j=1}^N\Big(\int_{\R^{dN}}\vphi
^{\frac{p}{2}}F_N\ud V  \Big)^{\frac{2}{p}}  \Big(\int_{\R^{dN}}|v^j|
^{2}F_N\ud V  \Big)^{\frac{p-2}{p}} \\
\leq &\Big(\int_{\R^{dN}}\vphi
^{\frac{p}{2}}F_N\ud V  \Big)^{\frac{2}{p}}  \Big(\frac{1}{N}\sum_{j=1}^N\int_{\R^{dN}}|v^j|
^{p}F_N\ud V  \Big)^{\frac{p-2}{p}} ,
\end{aligned}
$$
where the last step is due to the concavity of the $\frac{p-2}{p}$-th power. From \eqref{otherhand} we get
\begin{equation}\label{mp}
    \frac{\ud \mm_p(t)}{\ud t}\leq -p(d-1)\mm_p(t)+p(p+d-3)\Big(\int_{\R^{dN}}\vphi
^{\frac{p}{2}}F_N\ud V  \Big)^{\frac{2}{p}} [\mm_p(t)]^{\frac{p-2}{p}}.
\end{equation}
Now, we need to propagate the $\frac{p}{2}$-th moment of $\vphi(V)$, where $\nabla_{v^i}\vphi=\frac{2}{N} v^i$ and $\nabla^2_{v^i}\vphi=\frac{2}{N} \Id$. For any index $i$ and $j$, we have
$$
\nabla_{v^i}\vphi^\frac{p}{2}=\frac{p}{2}\vphi^{\frac{p}{2}-1}\nabla_{v^i}\vphi= \frac{p}{N}v^i\vphi^{\frac{p}{2}-1}\color{black}, \quad
\nabla_{v^j}\vphi^\frac{p}{2}= \frac{p}{N}v^j\vphi^{\frac{p}{2}-1}\color{black}, 
$$
and 
$$
\nabla^2_{v^i}\vphi^\frac{p}{2}=\frac{p}{2}\vphi^{\frac{p}{2}-1}\nabla^2_{v^i}\vphi+\frac{p(p-2)}{4}\vphi^{\frac{p}{2}-2}\nabla_{v^i}\vphi\otimes \nabla_{v^i}\vphi= \frac{p}{N}\vphi^{\frac{p}{2}-1}\Id\color{black}+\frac{p(p-2)}{N^2}\vphi^{\frac{p}{2}-2}v^i\otimes v^i.
$$
We propagate the integral of $\vphi^\frac{p}{2}$ with respect to the joint law $F_N$, using that
$$
\begin{aligned}
&\frac{1}{2}(\nabla_{v^i}^2\vphi+\nabla_{v^j}^2\vphi):a(v^i-v^j)-(\nabla_{v^i}\vphi-\nabla_{v^j}\vphi)\cdot b(v^i-v^j)\\
=& \frac{1}{2}\big(\frac{p}{N}\vphi^{\frac{p}{2}-1}\Id+\frac{p}{N}\vphi^{\frac{p}{2}-1}\Id\big):\big(\Id|v^i-v^j|^2-(v^i-v^j)\otimes (v^i-v^j)\big)\\&-(d-1)\big(\frac{p}{N}v^i\vphi^{\frac{p}{2}-1}-\frac{p}{N}v^j\vphi^{\frac{p}{2}-1}\big)\cdot (v^i-v^j)+\frac{p(p-2)}{N^2}\big(|v^i|^2|v^j|^2-(v^i\cdot v^j)^2\big)\vphi^{\frac{p}{2}-2}\\
=& \frac{p}{N}\vphi^{\frac{p}{2}-1}\Id:\big(\Id|v^i-v^j|^2-(v^i-v^j)\otimes (v^i-v^j)\big)-\frac{p}{N}\vphi^{\frac{p}{2}-1}(d-1)|v^i-v^j|^2\\&+\frac{p(p-2)}{N^2}\big(|v^i|^2|v^j|^2-(v^i\cdot v^j)^2\big)\vphi^{\frac{p}{2}-2}\\=&\frac{p(p-2)}{N^2}\big(|v^i|^2|v^j|^2-(v^i\cdot v^j)^2\big)\vphi^{\frac{p}{2}-2},
\end{aligned}
$$
to arrive at
\begin{equation}\label{estimate of phi}
\begin{aligned}
 &\frac{\ud }{\ud t}\int_{\R^{dN}}\vphi^\frac{p}{2} F_N(t,V)\ud V=\frac{\ud }{\ud t}\int_{\R^{dN}}\big(\frac{1}{N}\sum_{k=1}^{N}|v^k|^2\big)^\frac{p}{2} F_N(t,V)\ud V  \\
 =&\frac{1}{N}\sum_{i,j=1}^{N}\int_{\R^{dN}}\bigg[\frac{1}{2}\big(\nabla^2_{v^{i}}\varphi^{\frac{p}{2}}+\nabla^2_{v^{j}}\varphi^{\frac{p}{2}}  \big):a(v^{i}-v^j) \\&\qquad\qquad\qquad\qquad\qquad\qquad+\big(\nabla_{v^{i}} \varphi^{\frac{p}{2}}-\nabla_{v^{j}} \varphi^{\frac{p}{2}}\big)\cdot b(v^{i}-v^j)\bigg]F_N(t,V)\ud V\\
 =&\frac{p(p-2)}{N^3}\int_{\R^{dN}}\sum_{i,j=1}^{N}\big(|v^i|^2|v^j|^2-(v^i\cdot v^j)^2\big)\vphi^{\frac{p}{2}-2}F_N(t,V)\ud V\\
 \leq &\frac{p(p-2)}{N}\int_{\R^{dN}}\vphi^{\frac{p}{2}}F_N(t,V)\ud V.
\end{aligned}    
\end{equation}
The Gr\"onwall's inequality implies that
$$
\int_{\R^{dN}}\vphi^{\frac{p}{2}}F_N(t,V)\ud V\leq e^{\frac{p(p-2)}{N}t}\int_{\R^{dN}}\vphi^{\frac{p}{2}}F_N(0,V)\ud V,
$$
and by the Jensen's inequality with $p>2$ that
$$
\Big(\frac{1}{N}\sum_{i=1}^N|v_{\alpha}|^2\Big)^{\frac{p}{2}}\leq \frac{1}{N}\sum_{i=1}^N|v_{\alpha}|^p,
$$
hence we have that
$$
\int_{\R^{dN}}\vphi^{\frac{p}{2}}F_N(t,V)\ud V\leq e^{\frac{p(p-2)}{N}t}\int_{\R^{d}}|v|^pf_0(v)\ud v=e^{\frac{p(p-2)}{N}t}\mm_p(0).
$$
And substituting this into \eqref{mp}, we have 
\begin{equation}
    \frac{\ud \mm_p(t)}{\ud t}\leq -p(d-1)\mm_p(t)+p(p+d-3)e^{\frac{2(p-2)}{N}t}[\mm_p(0)]^{\frac{2}{p}} [\mm_p(t)]^{1-\frac{2}{p}}.
\end{equation}
Now we solve this ODE by
$$
\begin{aligned}
 \frac{\ud }{\ud t} \big(e^{2(d-1)t}[\mm_p(t)]^{\frac{2}{p}}\big) =& 2(d-1)e^{2(d-1)t}[\mm_p(t)]^{\frac{2}{p}}+\frac{2}{p}e^{2(d-1)t}[\mm_p(t)]^{\frac{2}{p}-1}\frac{\ud \mm_p(t)}{\ud t}\\
 =&\frac{2}{p}e^{2(d-1)t}[\mm_p(t)]^{\frac{2}{p}-1}\Big(p(d-1)\mm_p(t)+\frac{\ud \mm_p(t)}{\ud t}\Big)\\
 \leq &  2e^{2(d-1)t}(p+d-3)e^{\frac{2(p-2)}{N}t}\big[\mm_p(0) \big]^{\frac{2}{p}} .
\end{aligned}
$$
Then rewrite it into the integral form as
$$
\begin{aligned}
&e^{2(d-1)t}[\mm_p(t)]^{\frac{2}{p}}-[\mm_p(0)]^{\frac{2}{p}}\leq 2(p+d-3)\big[\mm_p(0) \big]^{\frac{2}{p}} \int_0^te^{2(d-1)s}e^{\frac{2(p-2)}{N}s}\ud s\\\leq &2(p+d-3)\big[\mm_p(0) \big]^{\frac{2}{p}} \frac{N}{2(d-1)N+2(p-2)}\Big(e^{2(d-1)t}e^{\frac{2(p-2)}{N}t}-1\Big),     
\end{aligned}
$$
i.e.
$$
\begin{aligned}
 \left[\mm_p(t)\right]^{\frac{2}{p}}&\leq [\mm_p(0)]^{\frac{2}{p}}\bigg(e^{-2(d-1)t}+\frac{(p+d-3)N}{(d-1)N+(p-2)}\Big(e^{\frac{2(p-2)}{N}t}-e^{-2(d-1)t}\Big)\bigg)  \\ 
 &\leq [\mm_p(0)]^{\frac{2}{p}}\Big(e^{-2(d-1)t}+\frac{p+d-3}{d-1}\big(e^{\frac{2(p-2)}{N}t}-e^{-2(d-1)t}\big)\Big)  \\ 
 &\leq [\mm_p(0)]^{\frac{2}{p}}\Big(\frac{p+d-3}{d-1}e^{\frac{2(p-2)}{N}t}\Big),\\ 
\end{aligned}
$$
where in the last step we used $\frac{p+d-3}{d-1}>1$ when $p>2 $.
That is,
$$
\mm_p(t)\leq \mm_p(0)\Big(\frac{p+d-3}{d-1}e^{\frac{2(p-2)}{N}t}\Big)^{\frac{p}{2}} =\mm_p(0)\left(\frac{p+d-3}{d-1}\right)^{\frac{p}{2}}e^{\frac{p(p-2)t}{N}}.
$$
In conclusion, we have proved the desired estimate.
\end{proof} 

\bigskip

\bibliographystyle{alpha}
\bibliography{ref}

\end{document}